\documentclass[12pt]{article}

\usepackage{latexsym,amsmath,amscd,amssymb,graphics,float}
\usepackage{enumerate}

\usepackage{graphicx,lscape}

\usepackage[colorlinks]{hyperref}
\usepackage{url}

\begin{document}

\newenvironment{proof}[1][Proof]{\textbf{#1.} }{\ \rule{0.5em}{0.5em}}

\newtheorem{theorem}{Theorem}[section]
\newtheorem{definition}[theorem]{Definition}
\newtheorem{lemma}[theorem]{Lemma}
\newtheorem{remark}[theorem]{Remark}
\newtheorem{proposition}[theorem]{Proposition}
\newtheorem{corollary}[theorem]{Corollary}
\newtheorem{example}[theorem]{Example}

\numberwithin{equation}{section}
\newcommand{\ep}{\varepsilon}
\newcommand{\R}{{\mathbb  R}}
\newcommand\C{{\mathbb  C}}
\newcommand\Q{{\mathbb Q}}
\newcommand\Z{{\mathbb Z}}
\newcommand{\N}{{\mathbb N}}

\newcommand{\bfi}{\bfseries\itshape}

\newsavebox{\savepar}
\newenvironment{boxit}{\begin{lrbox}{\savepar}
\begin{minipage}[b]{15.5cm}}{\end{minipage}\end{lrbox}
\fbox{\usebox{\savepar}}}

\title{{\bf Dynamical systems with a prescribed attracting set and applications to conservative dynamics}}
\author{R\u{a}zvan M. Tudoran}

\date{}
\maketitle \makeatother

\begin{abstract}
We provide an explicit method to construct dynamical systems which admit an a-priori prescribed attracting set. As application, we provide a method to construct perturbations of conservative dynamical systems, which admit an a-priori prescribed leafwise attracting set.
\end{abstract}

\medskip



\section{Introduction}
\label{section:one}

The aim of this article is to provide an explicit method to construct dynamical systems which admit an arbitrary a-priori prescribed attracting set, i.e., a closed and invariant set which attracts every bounded positive orbit of the dynamical system. As application, we give an answer to the following problem: given a conservative $n-$dimensional dynamical system (i.e., a dynamical system which admits a $(k+p)-$dimensional vector type first integral, where $k+p < n$) and an invariant set $\mathcal{S}$ (given as the level set of a $k-$dimensional first integral defined by some $k-$dimensional projection of the original $(k+p)-$dimensional first integral), construct a curve of dynamical systems starting from the original system, such that each system on this curve is still conservative (admitting the $p-$dimensional first integral which together with the $k-$dimensional first integral, forms the original $(k+p)-$dimensional first integral), keeps invariant the set $\mathcal{S}\cap \operatorname{Mrk}$ (where $\operatorname{Mrk}$ is the open set consisting of the points where the rank of the $(k+p)-$dimensional first integral is maximal), and moreover, the intersection of $\mathcal{S}\cap \operatorname{Mrk}$ with each level set (corresponding to regular values) of the $p-$dimensional first integral, is an attracting set of each system on the curve (except for the original system) restricted to the corresponding level set. 

More precisely, in the second section of this article we recall a result from \cite{tudoran} which provides an explicit method to construct the class of smooth vector fields (defined on a smooth Riemannian manifold) which admit an a-priori given set of first integrals and, in the same time, dissipate a given set of scalar quantities, with a-priori defined dissipation rates. The third section represent the main part of this work, and gives an explicit method to construct a dynamical system which admits an a-priori defined attracting set. The only requirement needed in order to construct the vector field associated to a given closed subset of $\mathbb{R}^n$, is to know a representation of this set as a level set of some smooth function. Consequently, since any closed subset of $\mathbb{R}^n$ can be expressed as a level set of some smooth function, this method makes possible the construction of a vector field which have as attracting set a general a-priori prescribed closed subset of $\mathbb{R}^n$. The fourth section gives an application of the results given in the previous section, to construct leafwise attracting sets for perturbations of conservative dynamical systems. More precisely, let $\mathfrak{S}$ be a given dynamical system (defined on an open subset $U\subseteq\mathbb{R}^n$) which admits $k+p$ smooth first integrals, $I_1,\dots, I_k, D_1,\dots, D_p$ (or equivalently, it admits two vector type first integrals, $I:=(I_1,\dots, I_k)$ and $D:=(D_1,\dots, D_p)$). Let $\Sigma^{D_1,\dots, D_p}_{d_1,\dots, d_p}$ be a dynamically invariant set of $\mathfrak{S}$, given by the level set of the vector type first integral $D$ corresponding to some (regular or singular) value $d:=(d_1,\dots,d_p)\in\operatorname{Im}(D)$. Starting with these data, we construct a smooth family of dynamical systems $\{\mathfrak{S}_{\lambda}\}_{\lambda \geq 0}$ (defined on the open subset $\operatorname{Mrk}((D,I))\subseteq U$ consisting of the points of maximum rank of the smooth function $(D,I)$), such that $\mathfrak{S}_{0} =\mathfrak{S}|_{\operatorname{Mrk}((D,I))}$, and for all $\lambda >0$, the associated dynamical system, $\mathfrak{S}_{\lambda}$, admits also the vector type first integral $I|_{\operatorname{Mrk}((D,I))}$, keeps dynamically invariant the set $\Sigma^{D_1,\dots, D_p}_{d_1,\dots, d_p}\cap\operatorname{Mrk}((D,I))$ and moreover, the invariant set $\Sigma^{D_1,\dots, D_p}_{d_1,\dots, d_p}\cap (I|_{\operatorname{Mrk}((D,I))})^{-1}(\{\mu\})$ (if not void) is an attracting set of $\mathfrak{S}_{\lambda}|_{(I|_{\operatorname{Mrk}((D,I))})^{-1}(\{\mu\})}$, for every regular value  $\mu\in\operatorname{Im}(I|_{\operatorname{Mrk}((D,I))})$. In particular, if $\mu$ is a regular value of $I|_{\operatorname{Mrk}((D,I))}$ such that the intersection of some connected components of $(I|_{\operatorname{Mrk}((D,I))})^{-1}(\{\mu\})$ with the invariant set $\Sigma^{D_1,\dots, D_p}_{d_1,\dots, d_p}\cap\operatorname{Mrk}((D,I))$, contain a single orbit of the dynamical system $\mathfrak{S}$, e.g., equilibrium point, periodic orbit, homoclinic or heteroclinic cycle (if any such $\mu$ exists), then these orbits preserve their nature as orbits of the dynamical system $\mathfrak{S}_{\lambda}$ (for each $\lambda >0$), and moreover they attract every bounded positive orbit of the dynamical system $\mathfrak{S}_{\lambda}|_{(I|_{\operatorname{Mrk}((D,I))})^{-1}(\{\mu\})}$ sharing the same connected component. In the case of equilibrium points, these become asymptotically stable, as equilibrium states of the dynamical system $\mathfrak{S}_{\lambda}|_{(I|_{\operatorname{Mrk}((D,I))})^{-1}(\{\mu\})}$. The aim of the last section of this article is to present the correspondents of the main results of the previous section, in the case when the invariant set $\Sigma^{D_1,\dots, D_p}_{d_1,\dots, d_p}\cap \operatorname{Mrk}((D,I))$ is foliated by the level sets of regular values of the vector type first integral $D^{p^{\prime}\rightarrow }|_{\operatorname{Mrk}((D,I))}:=({D_{p^{\prime}+1}}|_{\operatorname{Mrk}((D,I))},\dots, {D_{p}}|_{\operatorname{Mrk}((D,I))})$, where $p^{\prime}$ is a natural number, such that $0<p^{\prime}<p$.

\section{Dynamical systems with prescribed conserved and dissipated scalar quantities}

In this section we recall a result from \cite{tudoran} which provides a constructive method to obtain the class of smooth vector fields defined on a smooth Riemannian manifold, which admits an a-priori given set of first integrals and, in the same time, dissipate a given set of scalar quantities, with a-priori defined dissipation rates.

More precisely, we have the following result, which is the key ingredient to obtain the main results of this article.

\begin{theorem}[\cite{tudoran}]\label{MTD}
Let $(M,g)$ be an $n$-dimensional smooth Riemannian manifold, and fix $k,p\in \mathbb{N}$ two natural numbers such that $p>0$ and $0<k+p\leq n$. Let $h_1, \dots, h_p \in \mathcal{C}^{\infty}(U,\mathbb{R})$ be a given set of smooth functions defined on an open subset $U\subseteq M$, and respectively let $I_1,\dots, I_k, D_1,\dots, D_p\in \mathcal{C}^{\infty}(U,\mathbb{R})$ be given, such that $$\{\nabla_{g} I_1 ,\dots,\nabla_{g} I_k , \nabla_{g} D_1 , \dots,\nabla_{g} D_p\}\subset \mathfrak{X}(U)$$ form a set of pointwise linearly independent vector fields on $U$.

Then the set of solutions $X\in \mathfrak{X}(U)$ of the system
\begin{equation}\label{EQA}
\begin{split}
\left\{\begin{array}{l}
\mathcal{L}_{X}I_1= \dots = \mathcal{L}_{X}I_k = 0,\\
\mathcal{L}_{X}D_1 = h_1, \dots, \mathcal{L}_{X}D_p = h_p,\\
\end{array}\right.
\end{split}
\end{equation}
forms the affine distribution
$$
\mathfrak{A}[X_0;\nabla_{g} I_1 ,\dots,\nabla_{g} I_k , \nabla_{g} D_1 , \dots, \nabla_{g} D_p ]=X_0 + \mathfrak{X}[\nabla_{g} I_1,\dots, \nabla_{g} I_k, \nabla_{g} D_1, \dots, \nabla_{g} D_p],
$$
locally generated by the set of vector fields
$$
\left\{X_0\right\}\biguplus \left\{\star\left( \bigwedge_{i=1, i\neq a}^{n-(k+p)} Z_i \wedge \bigwedge_{j=1}^{p} \nabla_{g} D_j \wedge\bigwedge_{l=1}^{k} \nabla_{g} I_l \right): a\in\{1,\dots, n-(k+p)\right\}
$$
where
$$
X_{0}=\left\| \bigwedge_{i=1}^{p} \nabla_{g} D_i\wedge\bigwedge_{j=1}^{k} \nabla_{g} I_j \right\|_{k+p}^{-2}\cdot\sum_{i=1}^{p}(-1)^{n-i}h_i \Theta_i, 
$$
$$
\Theta_i = \star\left[ \bigwedge_{j=1, j\neq i}^{p} \nabla_{g} D_j \wedge \bigwedge_{l=1}^{k} \nabla_{g} I_l  \wedge\star\left(\bigwedge_{j=1}^{p} \nabla_{g} D_j\wedge\bigwedge_{l=1}^{k} \nabla_{g} I_l \right)\right],
$$
and respectively the set of locally defined vector fields $$\{\nabla_{g} I_1,\dots, \nabla_{g} I_k, \nabla_{g} D_1, \dots, \nabla_{g} D_p, Z_1, \dots, Z_{n-(k+p)}\}$$ forms a moving frame. The notation $\star$ stands for the Hodge star operator on multivector fields, and $\mathcal{L}_{X}F:=g(\nabla_{g} F, X)$ stands for the Lie derivative of the smooth function $F\in\mathcal{C}^{\infty}(U,\mathbb{R})$ along the vector field $X$. 
\end{theorem}
Note that in contrast with the vector fields $\nabla_{g} I_1 ,\dots,\nabla_{g} I_k , \nabla_{g} D_1 , \dots,\nabla_{g} D_p$, which are globally defined on $U$, the vector fields $Z_1, \dots, Z_{n-(k+p)}$ exist in general only locally around each point $x\in U$, in some open neighborhood $U_x \subseteq U$. Nevertheless, the equations \eqref{EQA} have a globally defined particular solution in $U$, given by the vector field $X_0$. Moreover, if $X$ is a vector field which conserves $I_1,\dots, I_k, D_1,\dots, D_p$ (i.e., $X$ is a solution of the homogeneous system associated to \eqref{EQA}), then $X+X_0$ is a global solution of \eqref{EQA}. More precisely, we have the following result from \cite{tudoran}.
\begin{theorem}[\cite{tudoran}]\label{thmdin}
Let $\dot x =X(x)$ be the dynamical system generated by a vector field $X\in\mathfrak{X}(U)$ which conserves the smooth functions $I_1,\dots, I_k, D_1,\dots, D_p\in \mathcal{C}^{\infty}(U,\mathbb{R})$. Assume that $\nabla_g I_1,\dots, \nabla_g I_k, \nabla_g D_1,\dots, \nabla_g D_p$ are pointwise linearly independent on some open subset $V\subseteq U$.

Then the perturbed dynamical system
$$
\dot x= X(x)+ X_{0}(x),~  x\in V,
$$
with $X_0 \in\mathfrak{X}(V)$ given in Theorem \eqref{MTD}, is a dissipative dynamical system, generated by the dissipative vector field $X+X_0 \in\mathfrak{X}(V)$ which conserves $I_1,\dots, I_k$ (i.e., $\mathcal{L}_{X+X_0}I_1= \dots = \mathcal{L}_{X+X_0}I_k = 0$) and dissipates $D_1,\dots, D_p$ with (corresponding) dissipation rates $h_1,\dots, h_p$ (i.e., $\mathcal{L}_{X+X_0}D_1=h_1,\dots,\mathcal{L}_{X+X_0}D_p = h_p$).
\end{theorem}

\section{Dynamical systems with prescribed attracting set}

This section is the main part of this paper and gives an explicit method to construct a dynamical system which admits an a-priori defined attracting set. The only requirement needed in order to construct the vector field associated to a given closed subset of $\mathbb{R}^n$, is to know a representation of this set as the level set of some smooth function. Consequently, since any closed subset of $\mathbb{R}^n$ can be expressed as a level set of some smooth function, this method makes possible the construction of a vector field which have as attracting set a general a-priori prescribed closed subset of $\mathbb{R}^n$.

Let us start by fixing a natural number $1\leq p\leq n$, and a closed subset of $\mathbb{R}^n$ given by $$\Sigma^{D_1,\dots,D_p}_{d_1,\dots,d_p}:=D^{-1}(\{(d_1,\dots,d_p)\})\subset U\subseteq \mathbb{R}^n,$$ where $U\subseteq \mathbb{R}^n$ is an open subset, $D:=(D_1,\dots,D_p) : U\rightarrow \mathbb{R}^p$ is a smooth function, and $(d_1,\dots,d_p)\in\mathbb{R}^p$ is some point in the image of $D$. Note that if $(d_1,\dots,d_p)\in\mathbb{R}^p$ is a regular value of $D$, then $\Sigma^{D_1,\dots,D_p}_{d_1,\dots,d_p}$ is a smooth $(n-p)$-dimensional submanifold of $\mathbb{R}^n$, and hence for every $x\in\Sigma^{D_1,\dots,D_p}_{d_1,\dots,d_p}$, the vectors $\nabla D_1 (x),\dots,\nabla D_p (x)$ are linearly independent, where $\nabla$ stands for the gradient operator associated with respect to the canonical inner product on $\mathbb{R}^n$. By Sard's theorem we know that almost all points in the image of $D$ are regular values, i.e., the set of singular values of $D$ is a set of Lebesgue measure zero in $\mathbb{R}^p$. Let us denote by $\operatorname{Mrk}(D)\subseteq U$ the set of maximal rank points of $D$, i.e., the points $x\in U$ such that the vectors $\nabla D_1 (x),\dots,\nabla D_p (x)$ are linearly independent. Recall that $\operatorname{Mrk}(D)$ is an open subset of $U$ contained in the set of regular points of $D$. Recall that a point $x_0 \in U$ is a regular point of $D$ if there exists an open neighborhood $U_{x_0}\subseteq U$ such that $ \operatorname{rank}(\mathrm{d}D (x))=\operatorname{rank}(\mathrm{d}D (x_0))$, for all $x\in U_{x_0}$. Recall also that the set of regular points of $D$ is an open dense subset of $U$ in contrast with $\operatorname{Mrk}(D)$ which is open but not necessarily dense. The rank of $\mathrm{d}D (\cdot)$ is constant on each connected component of the set of regular points of $D$. Concerning the set $\Sigma^{D_1,\dots,D_p}_{d_1,\dots,d_p}$, if $(d_1,\dots,d_p)$ is a regular value of $D$, then $\Sigma^{D_1,\dots,D_p}_{d_1,\dots,d_p} \subset \operatorname{Mrk}(D)$.
\medskip

Before stating the main theorem of this section, let us recall briefly some terminology and also some classical results we shall need in the sequel. For details see e.g., \cite{hartman}, \cite{verhulst}. In order to do that, let us consider a smooth vector field $X\in\mathfrak{X}(U)$ defined on an open set $U\subseteq\mathbb{R}^n$. Then, for each $\overline{x}\in U$ we shall denote by $t\in I_{\overline{x}}\subseteq \mathbb{R}\mapsto x(t;\overline{x})\in U$ the integral curve of $X$ starting from $\overline{x}$ at $t=0$, i.e., the solution of the Cauchy problem ${\mathrm{d}x}/{\mathrm{d}t}=X(x(t))$, $x(0)=\overline{x}$, defined on the maximal domain $I_{\overline{x}}\subseteq \mathbb{R}$, where $I_{\overline{x}}$ is an open interval of $\mathbb{R}$ containing the origin. For each $\overline{x}\in U$ we associate the set $\mathcal{O}^{+}_{\overline{x}}:=\{y\in U : y=x(t;\overline{x}), ~ t\geq 0\}$, called \textit{the positive orbit} of $\overline{x}$. Consequently, a subset $\mathcal{S}\subseteq U$ is called \textit{positively invariant} if for every $\overline{x}\in\mathcal{S}$ we have that $\mathcal{O}^{+}_{\overline{x}}\subseteq \mathcal{S}$. If a set $\mathcal{S}$ is positively invariant, then so are the sets $\overline{\mathcal{S}}$ and $\mathring{\mathcal{S}}$. Let us recall that if $\mathcal{O}^{+}_{\overline{x}}$ is contained in some compact subset of $U$, then the solution $x(t;\overline{x})$ is defined for all $t\in[0,\infty)$. If one denotes by $\omega(\overline{x}):=\{y\in U: (\exists)(t_{n})_{n\in\mathbb{N}}\subset [0,\infty),~t_{n} < t_{n+1},~{t_n}\rightarrow {\infty}~\text{s.t.}~\lim_{n\rightarrow \infty}x(t_n;\overline{x})=y\}$ the $\omega-$\textit{limit set} of $\overline{x}$, then we have that $\overline{\mathcal{O}^{+}_{\overline{x}}}=\mathcal{O}^{+}_{\overline{x}}\cup\omega(\overline{x})$, and $\omega(\overline{x})=\omega(x(t;\overline{x}))$, for all $t\geq 0$. Note that for all points $y\in\omega(\overline{x})$, we have that $\mathcal{O}^{+}_{y}\subseteq \omega(\overline{x})$, and hence the $\omega-$limit set of $\overline{x}$ can be expressed as $\omega(\overline{x})=\bigcap\{\overline{\mathcal{O}^{+}_{y}}: y\in\mathcal{O}^{+}_{\overline{x}}\}$. Moreover, for every $\overline{x}\in U$ such that the set $\{x(t;\overline{x}): t\geq 0\}$ is bounded, the associated $\omega-$limit set, $\omega(\overline{x})$, is a nonempty, invariant, compact and connected subset of $U$, and $x(t;\overline{x})$ approaches $\omega(\overline{x})$ for $t\rightarrow \infty$, i.e., $x(t;\overline{x})\rightarrow \omega(\overline{x})$ for $t\rightarrow \infty$. Recall that given a closed and invariant set $\mathcal{S}\subset U$, we say that the solution starting from a point $\overline{x}\in U$ \textit{approaches} the set $\mathcal{S}$ (and we denote $x(t;\overline{x})\rightarrow \mathcal{S}$), if for every $\varepsilon >0$ there exists $T>0$ such that $\operatorname{dist}(x(t;\overline{x}),\mathcal{S})<\varepsilon$, for all $t>T$, where for every point $p\in U$, $\operatorname{dist}(p,\mathcal{S}):=\inf_{x\in \mathcal{S}}\operatorname{dist}(p,x)$. In what follows, in order to show that a solution starting from a point $\overline{x}\in U$ approaches some closed and invariant set $\mathcal{S}$ as $t\rightarrow \infty$, we shall prove that $\omega(\overline{x})\subseteq \mathcal{S}$, and then using the attracting property of an $\omega-$limit set, i.e., $x(t;\overline{x})\rightarrow \omega(\overline{x})$ for $t\rightarrow \infty$, we get that $x(t;\overline{x})\rightarrow \mathcal{S}$ for $t\rightarrow \infty$.

\begin{definition}
Let $U\subset\mathbb{R}^n$ be an open subset of $\mathbb{R}^n$ and let $X\in\mathfrak{X}(U)$ be a smooth vector field that admits at least one bounded positive orbit. A closed and invariant subset $\mathcal{A}\subset U$ will be called \textbf{attracting set} of the dynamical system generated by $X$ if for \textbf{every} point $\overline{x}\in U$ such that the set $\{x(t;\overline{x}): t\geq 0\}$ is bounded, the integral curve of $X$ starting from $\overline{x}$ approaches $\mathcal{A}$ as $t\rightarrow \infty$, i.e., $x(t;\overline{x})\rightarrow \mathcal{A}$ for $t\rightarrow \infty$.
\end{definition}
Next result points out an important property of attracting sets.
\begin{remark}\label{rem1}
Let $\mathcal{A}\subset U$ be an attracting set of the dynamical system generated by a smooth vector field $X\in\mathfrak{X}(U)$. Then for every positively invariant compact set $K\subset U$ (if any), and every point $\overline{x}\in K$, the integral curve of $X$ starting from $\overline{x}$ approaches $\mathcal{A}$ as $t\rightarrow \infty$.
\end{remark}

In the above settings, we shall construct a smooth vector field $X$ defined on the open set $\operatorname{Mrk}(D)\subseteq U$, whose set of equilibrium points equals $\Sigma^{D_1,\dots,D_p}_{d_1,\dots,d_p}\cap \operatorname{Mrk}(D)$, and moreover, this set is an attracting set of $X$, i.e., for every $\overline{x} \in\operatorname{Mrk}(D)$, such that the set $\{x(t;\overline{x}): t\geq 0\}$ is bounded, we have that $x(t;\overline{x})\rightarrow \Sigma^{D_1,\dots,D_p}_{d_1,\dots,d_p}\cap \operatorname{Mrk}(D)$ for $t\rightarrow \infty$. Note that if $(d_1,\dots,d_p)$ is a regular value of $D=(D_1,\dots,D_p)$, then $\Sigma^{D_1,\dots,D_p}_{d_1,\dots,d_p} \subset \operatorname{Mrk}(D)$, and hence in this case, the vector field $X$ admits the attracting set $\Sigma^{D_1,\dots,D_p}_{d_1,\dots,d_p}$.

In order to do that, let us fix a strictly positive real number $\lambda >0$. Then using the Theorem \eqref{MTD}, we construct a smooth vector field $X\in\mathfrak{X}(\operatorname{Mrk}(D))$ such that 
\begin{equation}\label{SVF}
\mathcal{L}_{X}D_1 = (-\lambda)(D_1 - d_1), \dots, \mathcal{L}_{X}D_p = (-\lambda)(D_p - d_p).
\end{equation}
Note that by construction, $\Sigma^{D_1,\dots,D_p}_{d_1,\dots,d_p}\cap \operatorname{Mrk}(D)$ is a dynamically invariant set of $X$. By Theorem \eqref{MTD}, a particular solution of the equation \eqref{SVF} with maximal domain of definition, is given by the vector field $X_{0}^{\lambda} \in\mathfrak{X}(\operatorname{Mrk}(D))$ defined as follows:
\begin{equation*}
X_{0}^{\lambda}=\left\| \bigwedge_{i=1}^{p} \nabla D_i\right\|_{p}^{-2}\cdot\sum_{i=1}^{p}(-1)^{n-i}(-\lambda)(D_i - d_i)\Theta_i, 
\end{equation*}
where
$$
\Theta_i = \star\left[ \bigwedge_{j=1, j\neq i}^{p} \nabla D_j \wedge\star\left(\bigwedge_{j=1}^{p} \nabla D_j\right)\right].
$$

Let us state now the main result of this section.
\begin{theorem}\label{mainTHM}
Let $\mathcal{S}\subset\mathbb{R}^n$ be a nonempty closed subset of $\mathbb{R}^n$. Let $1\leq p\leq n$ be a natural number, $U\subseteq \mathbb{R}^n$ an open subset of $\mathbb{R}^n$, and let $D:=(D_1,\dots,D_p) : U\rightarrow \mathbb{R}^p$ be a smooth function such that $\mathcal{S}=\Sigma^{D_1,\dots,D_p}_{d_1,\dots,d_p}:=D^{-1}(\{(d_1,\dots,d_p)\})\subset U\subseteq \mathbb{R}^n$, for some point $(d_1,\dots,d_p)\in\mathbb{R}^p$ in the image of $D$. Assume that $\Sigma^{D_1,\dots,D_p}_{d_1,\dots,d_p}\cap \operatorname{Mrk}(D)\neq \emptyset$.

Then for each real number $\lambda >0$, one can associate a smooth vector field $X_{0}^{\lambda} \in\mathfrak{X}(\operatorname{Mrk}(D))$ given by  
\begin{equation}\label{XO}
X_{0}^{\lambda}=\left\| \bigwedge_{i=1}^{p} \nabla D_i\right\|_{p}^{-2}\cdot\sum_{i=1}^{p}(-1)^{n-i}(-\lambda)(D_i - d_i)\Theta_i, 
\end{equation}
where
$$
\Theta_i = \star\left[ \bigwedge_{j=1, j\neq i}^{p} \nabla D_j \wedge\star\left(\bigwedge_{j=1}^{p} \nabla D_j\right)\right], ~ i\in\{1,\dots,p\},
$$
such that the following statements hold true.
\begin{itemize}
\item [(a)] The set of the equilibrium states of the vector field $X_{0}^{\lambda}\in\mathfrak{X}(\operatorname{Mrk}(D))$, i.e., $\mathcal{E}(X_{0}^{\lambda}):=\{x\in\operatorname{Mrk}(D) : X_{0}^{\lambda} (x)=0\}$, is given by $\mathcal{E}(X_{0}^{\lambda})=\Sigma^{D_1,\dots,D_p}_{d_1,\dots,d_p}\cap \operatorname{Mrk}(D)$.
\item [(b)] The vector field $X_{0}^{\lambda}\in\mathfrak{X}(\operatorname{Mrk}(D))$ admits the attracting set  $\Sigma^{D_1,\dots,D_p}_{d_1,\dots,d_p}\cap \operatorname{Mrk}(D)$. More precisely, for every $\overline{x} \in\operatorname{Mrk}(D)$, such that the set $\{x(t;\overline{x}): t\geq 0\}$ is bounded, $x(t;\overline{x})\rightarrow \Sigma^{D_1,\dots,D_p}_{d_1,\dots,d_p}\cap \operatorname{Mrk}(D)$ for $t\rightarrow \infty$.
\end{itemize}
\end{theorem}
\begin{proof}
Let us define the smooth function $F:\operatorname{Mrk}(D)\rightarrow [0,\infty)$ given by
$$
F(x):=(D_1 (x)-d_1)^2 +\dots + (D_p (x)-d_p)^2, ~ (\forall) x\in\operatorname{Mrk}(D).
$$
Let $\overline{x}\in \operatorname{Mrk}(D)$ be given, and let $t\in I_{\overline{x}}\subseteq\mathbb{R}\mapsto x(t;\overline{x})\in\operatorname{Mrk}(D)$ be the integral curve of the vector field $X_{0}^{\lambda}\in\mathfrak{X}(\operatorname{Mrk}(D))$ such that $x(0;\overline{x})=\overline{x}$, where $I_{\overline{x}}\subseteq\mathbb{R}$ stands for the maximal domain of definition of the solution $x(\cdot;\overline{x})$.

Since the vector field $X_{0}^{\lambda}\in\mathfrak{X}(\operatorname{Mrk}(D))$ satisfies by construction the relations \eqref{SVF}, we have that
\begin{align}\label{lder}
\begin{split}
\mathcal{L}_{X_{0}^{\lambda}}F &= \sum_{i=1}^{p}\mathcal{L}_{X_{0}^{\lambda}}(D_i - d_i)^2 = \sum_{i=1}^{p}2(D_i - d_i)\mathcal{L}_{X_{0}^{\lambda}}(D_i - d_i)\\
&= \sum_{i=1}^{p}2(D_i - d_i)(-\lambda) (D_i - d_i)= (-2\lambda)\sum_{i=1}^{p}(D_i - d_i)^2 \\
&= (-2\lambda) F.
\end{split}
\end{align}
Using the relation \eqref{lder}, we obtain that
\begin{equation*}
\dfrac{\mathrm{d}}{\mathrm{d}t}F(x(t;\overline{x}))= (-2\lambda) F(x(t;\overline{x})), ~ (\forall)t\in I_{\overline{x}},
\end{equation*}
and hence 
\begin{equation}\label{grw}
F(x(t;\overline{x}))= \exp(-2\lambda t) \cdot F(\overline{x}), ~ (\forall)t\in I_{\overline{x}}.
\end{equation}
Moreover, since the set of zeros of $F$ coincides with $\Sigma^{D_1,\dots,D_p}_{d_1,\dots,d_p}\cap \operatorname{Mrk}(D)$, the following sets equality holds true:
\begin{equation}\label{zeroset}
\{x\in\operatorname{Mrk}(D) : (\mathcal{L}_{X_{0}^{\lambda}}F)(x)=0\}=\Sigma^{D_1,\dots,D_p}_{d_1,\dots,d_p}\cap \operatorname{Mrk}(D).
\end{equation}
\begin{itemize}
\item [(a)] In order to prove that $\Sigma^{D_1,\dots,D_p}_{d_1,\dots,d_p}\cap \operatorname{Mrk}(D)=\mathcal{E}(X_{0}^{\lambda})$, note that from the definition of $X_{0}^{\lambda}$ one obtains directly that $\Sigma^{D_1,\dots,D_p}_{d_1,\dots,d_p}\cap \operatorname{Mrk}(D) \subseteq \mathcal{E}(X_{0}^{\lambda})$. For proving the converse inclusion, let $\overline{x}\in \mathcal{E}(X_{0}^{\lambda})$ be an equilibrium point of $X_{0}^{\lambda}$. Hence, the associated integral curve satisfies the relation $x(t;\overline{x})=\overline{x}$, for all $t\in\mathbb{R}$, and so by relation \eqref{grw} we get that $F(\overline{x})=0$. Since $F(\overline{x})=\sum_{i=1}^{p}(D_i(\overline{x})-d_i)^2 =0$, and $\overline{x}\in\operatorname{Mrk}(D)$ by definition, we get that $\overline{x}\in\Sigma^{D_1,\dots,D_p}_{d_1,\dots,d_p}\cap \operatorname{Mrk}(D)$, and so as $\overline{x}$ was arbitrary chosen from $\mathcal{E}(X_{0}^{\lambda})$, we obtained also the converse inclusion, i.e., $\mathcal{E}(X_{0}^{\lambda})\subseteq \Sigma^{D_1,\dots,D_p}_{d_1,\dots,d_p}\cap \operatorname{Mrk}(D)$.
\item [(b)] In order to prove the second item, let $\overline{x}\in\operatorname{Mrk}(D)$ be an arbitrary point of the open set $\operatorname{Mrk}(D)$ such that the set $\{x(t;\overline{x}): t\geq 0\}$ is bounded. We shall show now that the $\omega-$limit set $\omega(\overline{x})$ is a subset of  $\Sigma^{D_1,\dots,D_p}_{d_1,\dots,d_p}\cap \operatorname{Mrk}(D)$. Indeed, let $y\in\omega(\overline{x})$ be arbitrary chosen. Then, there exists an increasing sequence $(t_{n})_{n\in\mathbb{N}}\subset [0,\infty)$, $\lim_{n\rightarrow \infty}t_n =\infty$, such that $\lim_{n\rightarrow \infty}x(t_n;\overline{x})=y$. Since the set $\{x(t;\overline{x}): t\geq 0\}$ is bounded, we get that $[0,\infty)\subset I_{\overline{x}}$, and hence the relation \eqref{grw} implies that
\begin{equation}\label{grwok}
F(x(t;\overline{x}))= \exp(-2\lambda t) \cdot F(\overline{x}), ~ (\forall)t\in [0,\infty).
\end{equation}
Consequently, for $t=t_n \geq 0$, $n\in\mathbb{N}$, the equality \eqref{grwok} becomes
\begin{equation*}
F(x(t_n;\overline{x}))= \exp(-2\lambda t_n) \cdot F(\overline{x}), ~ (\forall)n\in \mathbb{N}.
\end{equation*}
Since $\lim_{n\rightarrow \infty} t_n =\infty $, $\lambda >0$, $\lim_{n\rightarrow \infty}x(t_n;\overline{x})=y$, and $F$ is continuous, we obtain that $F(y)=0$, and hence taking into account that the set of zeros of $F$ is $\Sigma^{D_1,\dots,D_p}_{d_1,\dots,d_p}\cap \operatorname{Mrk}(D)$, it follows that $y\in \Sigma^{D_1,\dots,D_p}_{d_1,\dots,d_p}\cap \operatorname{Mrk}(D)$. As $y\in\omega(\overline{x})$ was arbitrary chosen, we obtain that $\omega(\overline{x})\subseteq \Sigma^{D_1,\dots,D_p}_{d_1,\dots,d_p}\cap \operatorname{Mrk}(D)$. Since $x(t;\overline{x})\rightarrow \omega(\overline{x})\subseteq \Sigma^{D_1,\dots,D_p}_{d_1,\dots,d_p}\cap \operatorname{Mrk}(D)$ for $t\rightarrow \infty$, it follows that $x(t;\overline{x})\rightarrow \Sigma^{D_1,\dots,D_p}_{d_1,\dots,d_p}\cap \operatorname{Mrk}(D)$ for $t\rightarrow \infty$.
\end{itemize}
\end{proof}

Using the properties of $\omega-$limit sets and the Theorem \eqref{mainTHM} we get the following result.
\begin{proposition}\label{propok}
In the hypothesis of Theorem \eqref{mainTHM}, the following assertions hold true:
\item [(a)] For every positively invariant compact set $K\subset \operatorname{Mrk}(D)$, and for every $\overline{x}\in K$ we have that $\omega(\overline{x})\subseteq \Sigma^{D_1,\dots,D_p}_{d_1,\dots,d_p}\cap \operatorname{Mrk}(D)$, and consequently $x(t;\overline{x})\rightarrow \Sigma^{D_1,\dots,D_p}_{d_1,\dots,d_p}\cap \operatorname{Mrk}(D)$ for $t\rightarrow \infty$.
\item [(b)] If $\Sigma^{D_1,\dots,D_p}_{d_1,\dots,d_p}\cap \operatorname{Mrk}(D)$ contains isolated points, then each such a point $\overline{x}\in \Sigma^{D_1,\dots,D_p}_{d_1,\dots,d_p}\cap\operatorname{Mrk}(D)$ is an asymptotically stable equilibrium point of $X_{0}^{\lambda}$.
\end{proposition}
\begin{proof}
The item $(a)$ follows directly from Theorem \eqref{mainTHM} and the Remark \eqref{rem1}. In order to prove the statement $(b)$ it is enough to construct a strict Lyapunov function associated to each isolated equilibrium state $x_e \in\Sigma^{D_1,\dots,D_p}_{d_1,\dots,d_p}\cap\operatorname{Mrk}(D)$ (recall from Theorem \eqref{mainTHM} that the set of equilibrium points of $X_{0}^{\lambda}$ equals to $\Sigma^{D_1,\dots,D_p}_{d_1,\dots,d_p}\cap\operatorname{Mrk}(D)$). In order to do that, pick an isolated point $x_e \in \Sigma^{D_1,\dots,D_p}_{d_1,\dots,d_p}\cap\operatorname{Mrk}(D)$. Hence, there exists $U_{x_e} \subseteq \operatorname{Mrk}(D)$, an open neighborhood of $x_e$, such that $U_{x_e}\cap(\Sigma^{D_1,\dots,D_p}_{d_1,\dots,d_p}\cap\operatorname{Mrk}(D))=\{x_e \}$, since $x_e$ is an isolated point of  $\Sigma^{D_1,\dots,D_p}_{d_1,\dots,d_p}\cap\operatorname{Mrk}(D)$.  

Let us define now the smooth function $F: U_{x_e}\rightarrow [0,\infty)$ given by
$$
F(x):=(D_1 (x)-d_1)^2 +\dots + (D_p (x)-d_p)^2, ~ (\forall) x\in U_{x_e}.
$$
Since by hypothesis we have that $U_{x_e}\cap(\Sigma^{D_1,\dots,D_p}_{d_1,\dots,d_p}\cap \operatorname{Mrk}(D))=\{x_e\}$, and the set of zeros of $F$ in $U_{x_e}$ is the set $U_{x_e}\cap(\Sigma^{D_1,\dots,D_p}_{d_1,\dots,d_p}\cap \operatorname{Mrk}(D))$, it follows that $x_e$ is the unique solution of the equation $F(x)=0$ in $U_{x_e}$.
Recall from \eqref{lder} that
\begin{equation}\label{ecok} 
(\mathcal{L}_{X_{0}^{\lambda}}F) (x)= (-2\lambda)F(x), ~(\forall)x\in U_{x_e}.
\end{equation}
Hence, we get that $F(x_e)=0$, $F(x)>0$, $(\mathcal{L}_{X_{0}^{\lambda}}F) (x)<0$, for every $x\in U_{x_e}\setminus\{x_e\}$, and consequently $F$ is a strict Lyapunov function associated to the equilibrium point $x_e$.
\end{proof}

\medskip
Next result is a reformulation of Theorem \eqref{mainTHM} in the case when $\Sigma^{D_1,\dots,D_p}_{d_1,\dots,d_p}$ is a closed subset of $\mathbb{R}^n$ given by $\Sigma^{D_1,\dots,D_p}_{d_1,\dots,d_p}:=D^{-1}(\{(d_1,\dots,d_p)\})\subset U\subseteq \mathbb{R}^n$, where $U\subseteq \mathbb{R}^n$ is an open set, $D:=(D_1,\dots,D_p) : U\rightarrow \mathbb{R}^p$ is a smooth function, and $(d_1,\dots,d_p)\in\mathbb{R}^p$ is a \textbf{regular value} of $D$. In this case recall that $\Sigma^{D_1,\dots,D_p}_{d_1,\dots,d_p} \subset \operatorname{Mrk}(D)$, and hence $\Sigma^{D_1,\dots,D_p}_{d_1,\dots,d_p}\cap \operatorname{Mrk}(D)=\Sigma^{D_1,\dots,D_p}_{d_1,\dots,d_p}$. Moreover, in this case,  $\Sigma^{D_1,\dots,D_p}_{d_1,\dots,d_p}$ is a $(n-p)-$dimensional smooth submanifold of $\mathbb{R}^n$.

\begin{theorem}\label{mainTHMreg}
Let $\mathcal{S}\subset\mathbb{R}^n$ be a nonempty closed subset of $\mathbb{R}^n$. Let $1\leq p\leq n$ be a natural number, $U\subseteq \mathbb{R}^n$ an open subset of $\mathbb{R}^n$, and let $D:=(D_1,\dots,D_p) : U\rightarrow \mathbb{R}^p$ be a smooth function such that $\mathcal{S}=\Sigma^{D_1,\dots,D_p}_{d_1,\dots,d_p}:=D^{-1}(\{(d_1,\dots,d_p)\})\subset U\subseteq \mathbb{R}^n$, for some point $(d_1,\dots,d_p)\in\mathbb{R}^p$ in the image of $D$ which happens to a be a regular value of $D$. 

Then for each real number $\lambda >0$, one can associate a smooth vector field $X_{0}^{\lambda} \in\mathfrak{X}(\operatorname{Mrk}(D))$ given by  
\begin{equation}\label{XOO}
X_{0}^{\lambda}=\left\| \bigwedge_{i=1}^{p} \nabla D_i\right\|_{p}^{-2}\cdot\sum_{i=1}^{p}(-1)^{n-i}(-\lambda)(D_i - d_i)\Theta_i, 
\end{equation}
where
$$
\Theta_i = \star\left[ \bigwedge_{j=1, j\neq i}^{p} \nabla D_j \wedge\star\left(\bigwedge_{j=1}^{p} \nabla D_j\right)\right], ~ i\in\{1,\dots,p\},
$$
such that the following statements hold true.
\begin{itemize}
\item [(a)] The set of the equilibrium states of the vector field $X_{0}^{\lambda}\in\mathfrak{X}(\operatorname{Mrk}(D))$, i.e., $\mathcal{E}(X_{0}^{\lambda}):=\{x\in\operatorname{Mrk}(D) : X_{0}^{\lambda} (x)=0\}$, is given by $\mathcal{E}(X_{0}^{\lambda})=\Sigma^{D_1,\dots,D_p}_{d_1,\dots,d_p}$.
\item [(b)] The vector field $X_{0}^{\lambda}\in\mathfrak{X}(\operatorname{Mrk}(D))$ admits the attracting set $\Sigma^{D_1,\dots,D_p}_{d_1,\dots,d_p}$. More precisely, for every $\overline{x} \in\operatorname{Mrk}(D)$, such that the set $\{x(t;\overline{x}): t\geq 0\}$ is bounded, $x(t;\overline{x})\rightarrow \Sigma^{D_1,\dots,D_p}_{d_1,\dots,d_p}$ for $t\rightarrow \infty$.
\end{itemize}
\end{theorem}

Using the properties of $\omega-$limit sets and the Theorem \eqref{mainTHMreg} we get the following result.
\begin{proposition}
In the hypothesis of Theorem \eqref{mainTHM}, the following assertions hold true:
\item [(a)] For every positively invariant compact set $K\subset \operatorname{Mrk}(D)$, and for every $\overline{x}\in K$ we have that $\omega(\overline{x})\subseteq \Sigma^{D_1,\dots,D_p}_{d_1,\dots,d_p}$, and consequently $x(t;\overline{x})\rightarrow \Sigma^{D_1,\dots,D_p}_{d_1,\dots,d_p}$ for $t\rightarrow \infty$.
\item [(b)] If $\Sigma^{D_1,\dots,D_p}_{d_1,\dots,d_p}$ contains isolated points, then each such a point $\overline{x}\in \Sigma^{D_1,\dots,D_p}_{d_1,\dots,d_p}$ is an asymptotically stable equilibrium point of $X_{0}^{\lambda}$.
\end{proposition}

\section{Application to conservative dynamics}

The aim of this section is to apply the main results of the above section in order to provide an answer to the following problem: given a conservative $n-$dimensional dynamical system (i.e., a dynamical system which admits a $(k+p)-$dimensional vector type first integral, where $k+p < n$; for a brief introduction see, e.g., \cite{ratiurazvan}, \cite{TudoranJGM}) and an invariant set $\mathcal{S}$ (given as the level set of a $k-$dimensional first integral defined by some $k-$dimensional projection of the original $(k+p)-$dimensional first integral), construct a curve of dynamical systems starting from the original system, such that each system on this curve is still conservative (admitting the $p-$dimensional first integral which together with the $k-$dimensional first integral, forms the original $(k+p)-$dimensional first integral), keeps invariant the set $\mathcal{S}\cap \operatorname{Mrk}$ (where $\operatorname{Mrk}$ is the open set consisting of the points where the rank of the $(k+p)-$dimensional first integral is maximal), and moreover, the intersection of $\mathcal{S}\cap \operatorname{Mrk}$ with each level set (corresponding to regular values) of the $p-$dimensional first integral, is an attracting set of each system on the curve (excepting the original system).

More precisely, let $\mathfrak{S}$ be a given dynamical system (defined on an open subset $U\subseteq\mathbb{R}^n$) which admits $k+p$ smooth first integrals, $I_1,\dots, I_k, D_1,\dots, D_p$ (or equivalently, it admits two vector type first integrals, $I:=(I_1,\dots, I_k)$ and $D:=(D_1,\dots, D_p)$). Let $\Sigma^{D_1,\dots, D_p}_{d_1,\dots, d_p}$ be a dynamically invariant set of $\mathfrak{S}$, given by the level set of the vector type first integral $D$ corresponding to some regular (or singular) value $d:=(d_1,\dots,d_p)\in\operatorname{Im}(D)$. Starting with these data, we construct a smooth family of dynamical systems $\{\mathfrak{S}_{\lambda}\}_{\lambda \geq 0}$ (defined on the open subset $\operatorname{Mrk}((D,I))\subseteq U$ consisting of the points of maximum rank of the smooth function $(D,I)$), such that $\mathfrak{S}_{0} =\mathfrak{S}|_{\operatorname{Mrk}((D,I))}$, and for all $\lambda >0$, the associated dynamical system, $\mathfrak{S}_{\lambda}$, admits also the vector type first integral $I|_{\operatorname{Mrk}((D,I))}$, keeps dynamically invariant the set $\Sigma^{D_1,\dots, D_p}_{d_1,\dots, d_p}\cap\operatorname{Mrk}((D,I))$ and moreover, the invariant set $\Sigma^{D_1,\dots, D_p}_{d_1,\dots, d_p}\cap (I|_{\operatorname{Mrk}((D,I))})^{-1}(\{\mu\})$ (if not void) is an attracting set of $\mathfrak{S}_{\lambda}|_{(I|_{\operatorname{Mrk}((D,I))})^{-1}(\{\mu\})}$, for every regular value  $\mu\in\operatorname{Im}(I|_{\operatorname{Mrk}((D,I))})$. In particular, if $\mu$ is a regular value of $I|_{\operatorname{Mrk}((D,I))}$ such that the intersection of some connected components of $(I|_{\operatorname{Mrk}((D,I))})^{-1}(\{\mu\})$ with the invariant set $\Sigma^{D_1,\dots, D_p}_{d_1,\dots, d_p}\cap\operatorname{Mrk}((D,I))$, contain a single orbit of the dynamical system $\mathfrak{S}$, e.g., equilibrium point, periodic orbit, homoclinic or heteroclinic cycle (if any such $\mu$ exists), then these orbits preserve their nature as orbits of the dynamical system $\mathfrak{S}_{\lambda}$ (for each $\lambda >0$), and moreover they attract every bounded positive orbit of the dynamical system $\mathfrak{S}_{\lambda}|_{(I|_{\operatorname{Mrk}((D,I))})^{-1}(\{\mu\})}$ sharing the same connected component. In the case of equilibrium points, these become asymptotically stable, as equilibrium states of the dynamical system $\mathfrak{S}_{\lambda}|_{(I|_{\operatorname{Mrk}((D,I))})^{-1}(\{\mu\})}$.

Let $\mathfrak{S}$ be a dynamical system, generated by a smooth vector field $X\in\mathfrak{X}(U)$ defined on an open subset $U\subseteq\mathbb{R}^n$, which admits $k+p$ smooth first integrals ($0< k+p < n$, $k\geq 0$), $I_1,\dots, I_k, D_1,\dots, D_p \in\mathcal{C}^{\infty}(U,\mathbb{R})$. In order to simplify the notations, we shall denote by $I$, $D$ and $(D,I)$, the vector type first integrals, $(I_1,\dots, I_k)$, $(D_1,\dots, D_p)$, and respectively $(D_1,\dots, D_p,I_1,\dots, I_k)$. Let $\mathcal{S}$ be a closed invariant set of $\mathfrak{S}$, given by the level set of the vector type first integral $D$ corresponding to some regular (or singular) value $d:=(d_1,\dots,d_p)\in\operatorname{Im}(D)$, i.e., $\mathcal{S}= \Sigma^{D_1,\dots, D_p}_{d_1,\dots, d_p}:=D^{-1}(\{(d_1,\dots,d_p)\})$.

In this settings, following mimetically the approach given in the previous section, we shall construct a family of smooth vector fields,  $X_{0}^{\lambda}$, $\lambda \geq 0$ (defined on the open set $\operatorname{Mrk}((D,I))\subseteq U$ consisting of those points of $U$ such that the rank of the differential of $(D,I):U\rightarrow \mathbb{R}^{p+k}$ evaluated at them is maximal), with properties similar to those of the vector field analyzed in Theorem \eqref{mainTHM}. More precisely, we shall prove that $X_{0}^{0}\equiv 0$, and for each $\lambda >0$, the vector field $X_{0}^{\lambda}\in\mathfrak{X}(\operatorname{Mrk}((D,I)))$ admits the vector type first integral $I|_{\operatorname{Mrk}((D,I))}$, keeps dynamically invariant the set $\Sigma^{D_1,\dots,D_p}_{d_1,\dots,d_p}\cap \operatorname{Mrk}((D,I))$, and moreover, each of the invariant sets $\Sigma^{D_1,\dots,D_p}_{d_1,\dots,d_p}\cap (I|_{\operatorname{Mrk}((D,I))})^{-1}(\{\mu\})$ (corresponding to regular values $\mu$ of $I|_{\operatorname{Mrk}((D,I))}$), is an attracting set for the dynamical system $\mathfrak{S}_{\lambda}|_{(I|_{\operatorname{Mrk}((D,I))})^{-1}(\{\mu\})}$.

In order to do that, let us fix a strictly positive real number $\lambda >0$. Then using the Theorem \eqref{MTD}, a particular solution of the system of equations 
\begin{equation}\label{SVFc}
\mathcal{L}_{X}I_1 =\dots = \mathcal{L}_{X}I_k =0, \mathcal{L}_{X}D_1 = (-\lambda)(D_1 - d_1), \dots, \mathcal{L}_{X}D_p = (-\lambda)(D_p - d_p),
\end{equation}
is given by the vector field $X= X_{0}^{\lambda} \in\mathfrak{X}(\operatorname{Mrk}((D,I)))$ defined by
\begin{equation}\label{xzero}
X_{0}^{\lambda}=\left\| \bigwedge_{i=1}^{p} \nabla D_i\wedge\bigwedge_{j=1}^{k} \nabla I_j \right\|_{k+p}^{-2}\cdot\sum_{i=1}^{p}(-1)^{n-i}(-\lambda) (D- d_i) \Theta_i, 
\end{equation}
where
$$
\Theta_i = \star\left[ \bigwedge_{j=1, j\neq i}^{p} \nabla D_j \wedge \bigwedge_{l=1}^{k} \nabla I_l  \wedge\star\left(\bigwedge_{j=1}^{p} \nabla D_j\wedge\bigwedge_{l=1}^{k} \nabla I_l \right)\right].
$$
Note that by construction, $X_{0}^{\lambda}\in\mathfrak{X}(\operatorname{Mrk}((D,I)))$ admits the vector type first integral $I|_{\operatorname{Mrk}((D,I))}:=({I_1}|_{\operatorname{Mrk}((D,I))},\dots,{I_k}|_{\operatorname{Mrk}((D,I))})$, keeps $\Sigma^{D_1,\dots,D_p}_{d_1,\dots,d_p}\cap \operatorname{Mrk}((D,I))$ dynamically invariant, as well as the set $\Sigma^{D_1,\dots,D_p}_{d_1,\dots,d_p}\cap (I|_{\operatorname{Mrk}((D,I))})^{-1}(\{\mu\})$, for each regular value $\mu$ of the vector type first integral $I|_{\operatorname{Mrk}((D,I))}$. Let us state now a theorem which points out some of the main properties of the vector field $X_{0}^{\lambda}\in\mathfrak{X}(\operatorname{Mrk}((D,I)))$.

\begin{theorem}\label{mainTHMconserv}
Let $X_{0}^{\lambda}\in\mathfrak{X}(\operatorname{Mrk}((D,I)))$ be the vector field defined by the relation \eqref{xzero}. Assuming that $\Sigma^{D_1,\dots,D_p}_{d_1,\dots,d_p}\cap \operatorname{Mrk}((D,I))\neq \emptyset$, the following statements hold true.
\begin{itemize}
\item [(a)] The set of the equilibrium states of the vector field $X_{0}^{\lambda}$, i.e., $\mathcal{E}(X_{0}^{\lambda}):=\{x\in\operatorname{Mrk}((D,I)) : X_{0}^{\lambda} (x)=0\}$, is given by $\mathcal{E}(X_{0}^{\lambda})=\Sigma^{D_1,\dots,D_p}_{d_1,\dots,d_p}\cap \operatorname{Mrk}((D,I))$.
\item [(b)] For each regular value $\mu$ of the vector type first integral $I|_{\operatorname{Mrk}((D,I))}$, the vector field ${X_{0}^{\lambda}}|_{(I|_{\operatorname{Mrk}((D,I))})^{-1}(\{\mu\})}$ admits the attracting set $\Sigma^{D_1,\dots,D_p}_{d_1,\dots,d_p}\cap (I|_{\operatorname{Mrk}((D,I))})^{-1}(\{\mu\})$. More precisely, for every $\overline{x} \in (I|_{\operatorname{Mrk}((D,I))})^{-1}(\{\mu\})$, such that the set $\{x(t;\overline{x}): t\geq 0\}$ is bounded, $x(t;\overline{x})\rightarrow \Sigma^{D_1,\dots,D_p}_{d_1,\dots,d_p}\cap (I|_{\operatorname{Mrk}((D,I))})^{-1}(\{\mu\})$ for $t\rightarrow \infty$.
\item [(c)] Let $\mu$ be an arbitrary fixed regular value of the vector type first integral $I|_{\operatorname{Mrk}((D,I))}$. Then for every positively invariant compact set $K\subset (I|_{\operatorname{Mrk}((D,I))})^{-1}(\{\mu\})$, and for every $\overline{x}\in K$ we have that $\omega(\overline{x})\subseteq \Sigma^{D_1,\dots,D_p}_{d_1,\dots,d_p}\cap (I|_{\operatorname{Mrk}((D,I))})^{-1}(\{\mu\})$, and consequently $x(t;\overline{x})\rightarrow \Sigma^{D_1,\dots,D_p}_{d_1,\dots,d_p}\cap (I|_{\operatorname{Mrk}((D,I))})^{-1}(\{\mu\})$ for $t\rightarrow \infty$. Particularly, if $I|_{\operatorname{Mrk}((D,I))}$ is a proper map, then the set $(I|_{\operatorname{Mrk}((D,I))})^{-1}(\{\mu\})$ is compact in $\operatorname{Mrk}((D,I))$, dynamically invariant, and consequently for every $\overline{x}\in (I|_{\operatorname{Mrk}((D,I))})^{-1}(\{\mu\})$, we have that $x(t;\overline{x})\rightarrow \Sigma^{D_1,\dots,D_p}_{d_1,\dots,d_p}\cap (I|_{\operatorname{Mrk}((D,I))})^{-1}(\{\mu\})$ for $t\rightarrow \infty$.
\item [(d)] Suppose $\mu$ is a regular value of $I|_{\operatorname{Mrk}((D,I))}$ such that  $\Sigma^{D_1,\dots,D_p}_{d_1,\dots,d_p}\cap (I|_{\operatorname{Mrk}((D,I))})^{-1}(\{\mu\})$ contains isolated points. Then each such a point $\overline{x}\in \Sigma^{D_1,\dots,D_p}_{d_1,\dots,d_p}\cap (I|_{\operatorname{Mrk}((D,I))})^{-1}(\{\mu\})$ is an asymptotically stable equilibrium point of $X_{0}^{\lambda}|_{(I|_{\operatorname{Mrk}((D,I))})^{-1}(\{\mu\})}$.
\end{itemize}
\end{theorem}
\begin{proof}
The proof follows mimetically those of Theorem \eqref{mainTHM} and Proposition \eqref{propok}.
\end{proof}

\medskip
At this point we have all necessary ingredients to construct a smooth family of dynamical systems $\{\mathfrak{S}_{\lambda}\}_{\lambda \geq 0}$ (defined on the open subset $\operatorname{Mrk}((D,I))\subseteq U$), such that $\mathfrak{S}_{0} =\mathfrak{S}|_{\operatorname{Mrk}((D,I))}$, and for all $\lambda >0$, the associated dynamical system, $\mathfrak{S}_{\lambda}$, admits the vector type first integral $I|_{\operatorname{Mrk}((D,I))}$, keeps dynamically invariant the set $\Sigma^{D_1,\dots,D_p}_{d_1,\dots,d_p}\cap\operatorname{Mrk}((D,I))$ and moreover, the invariant set $\Sigma^{D_1,\dots,D_p}_{d_1,\dots,d_p}\cap (I|_{\operatorname{Mrk}((D,I))})^{-1}(\{\mu\})$ (if not void) is an attracting set of $\mathfrak{S}_{\lambda}|_{(I|_{\operatorname{Mrk}((D,I))})^{-1}(\{\mu\})}$, for every regular value  $\mu\in\operatorname{Im}(I|_{\operatorname{Mrk}((D,I))})$. 

Before stating the main result of this section, recall that the dynamical system $\mathfrak{S}$ it was supposed to be generated by a smooth vector field $X\in\mathfrak{X}(U)$ which admits two vector type smooth first integrals, $D=(D_1,\dots,D_p):U \rightarrow \mathbb{R}^p$ and $I=(I_1,\dots,I_k):U \rightarrow \mathbb{R}^k$. In those hypothesis, we fixed a closed and invariant set, $\mathcal{S}\subset U$, given by the level set of $D$ corresponding to some regular (or singular) value $d:=(d_1,\dots,d_p)\in\operatorname{Im}(D)$, i.e., $\mathcal{S}= \Sigma^{D_1,\dots, D_p}_{d_1,\dots, d_p}:=D^{-1}(\{(d_1,\dots,d_p)\})$.

\begin{theorem}\label{THMcons1}
Let $X\in\mathfrak{X}(U)$ be a smooth vector field defined on an open subset $U\subseteq \mathbb{R}^n$, which admits $k+p$ smooth first integrals ($0< k+p < n$, $k\geq 0$), $I_1,\dots, I_k, D_1,\dots, D_p \in\mathcal{C}^{\infty}(U,\mathbb{R})$. Let $\mathcal{S}\subset U$ be a given closed and invariant set, defined as the level set of the vector type first integral $D:=(D_1,\dots, D_p)\in\mathcal{C}^{\infty}(U,\mathbb{R}^p)$ corresponding to some regular (or singular) value $d:=(d_1,\dots,d_p)\in\operatorname{Im}(D)$, i.e., $\mathcal{S}= \Sigma^{D_1,\dots, D_p}_{d_1,\dots, d_p}:=D^{-1}(\{(d_1,\dots,d_p)\})$. 

Then, to each $\lambda \geq 0$, we associate a smooth vector field, $X_{\lambda}:=X+ X_{0}^{\lambda}$, defined on the open set $\operatorname{Mrk}((D,I))\subseteq U$, where the smooth vector field $X_{0}^{\lambda} \in \mathfrak{X}(\operatorname{Mrk}((D,I)))$ is given by
\begin{equation*}
X_{0}^{\lambda}=\left\| \bigwedge_{i=1}^{p} \nabla D_i\wedge\bigwedge_{j=1}^{k} \nabla I_j \right\|_{k+p}^{-2}\cdot\sum_{i=1}^{p}(-1)^{n-i}(-\lambda) (D- d_i) \Theta_i, 
\end{equation*}
where
$$
\Theta_i = \star\left[ \bigwedge_{j=1, j\neq i}^{p} \nabla D_j \wedge \bigwedge_{l=1}^{k} \nabla I_l  \wedge\star\left(\bigwedge_{j=1}^{p} \nabla D_j\wedge\bigwedge_{l=1}^{k} \nabla I_l \right)\right].
$$
In the above settings, the following assertions hold true.
\begin{itemize}
\item[(a)] $\mathcal{E}(X_{\lambda})=\mathcal{E}(X)\bigcap\mathcal{E}(X^{\lambda}_{0})=\mathcal{E}(X)\bigcap\Sigma^{D_1,\dots, D_p}_{d_1,\dots, d_p}\bigcap\operatorname{Mrk}((D,I))$, where $\mathcal{E}(Z)$ stands for the set of equilibrium points of the vector field $Z$.
\item[(b)] The set $\Sigma^{D_1,\dots, D_p}_{d_1,\dots, d_p}\bigcap\operatorname{Mrk}((D,I))$ is a dynamically invariant set of the vector field $X_{\lambda}$, for every $\lambda \geq 0$. Moreover, for each regular value $\mu$ of the vector type first integral $$I|_{\operatorname{Mrk}((D,I))}:=({I_1}|_{\operatorname{Mrk}((D,I))},\dots,{I_k}|_{\operatorname{Mrk}((D,I))}),$$ 
the set $\Sigma^{D_1,\dots,D_p}_{d_1,\dots,d_p}\cap (I|_{\operatorname{Mrk}((D,I))})^{-1}(\{\mu\})$ is a closed and dynamically invariant set of the vector field ${X_{\lambda}}$, for every $\lambda \geq 0$.
\item [(c)] For each regular value $\mu$ of the vector type first integral $I|_{\operatorname{Mrk}((D,I))}$, the vector field ${X_{\lambda}}|_{(I|_{\operatorname{Mrk}((D,I))})^{-1}(\{\mu\})}$, $\lambda >0$, admits the attracting set $$\Sigma^{D_1,\dots,D_p}_{d_1,\dots,d_p}\cap (I|_{\operatorname{Mrk}((D,I))})^{-1}(\{\mu\}).$$ More precisely, for every $\overline{x} \in (I|_{\operatorname{Mrk}((D,I))})^{-1}(\{\mu\})$, such that the set $\{x(t;\overline{x}): t\geq 0\}$ is bounded, $x(t;\overline{x})\rightarrow \Sigma^{D_1,\dots,D_p}_{d_1,\dots,d_p}\cap (I|_{\operatorname{Mrk}((D,I))})^{-1}(\{\mu\})$ for $t\rightarrow \infty$. 
\item [(d)] Suppose there exists $\mu$, a regular value of $I|_{\operatorname{Mrk}((D,I))}$, such that the intersection of some connected component of $(I|_{\operatorname{Mrk}((D,I))})^{-1}(\{\mu\})$ with the invariant set $\Sigma^{D_1,\dots, D_p}_{d_1,\dots, d_p}\cap\operatorname{Mrk}((D,I))$, contain a single orbit $\gamma$ of $X$. Then $\gamma$ preserves its nature as an orbit of ${X_{\lambda}}|_{(I|_{\operatorname{Mrk}((D,I))})^{-1}(\{\mu\})}$ (for every $\lambda >0$), and moreover, attracts every bounded positive orbit of ${X_{\lambda}}|_{(I|_{\operatorname{Mrk}((D,I))})^{-1}(\{\mu\})}$, sharing the same connected component.
\item [(e)] Let $\mu$ be an arbitrary fixed regular value of the vector type first integral $I|_{\operatorname{Mrk}((D,I))}$. Then for every positively invariant compact set $K\subset (I|_{\operatorname{Mrk}((D,I))})^{-1}(\{\mu\})$, and for every $\overline{x}\in K$ we have that $\omega(\overline{x})\subseteq \Sigma^{D_1,\dots,D_p}_{d_1,\dots,d_p}\cap (I|_{\operatorname{Mrk}((D,I))})^{-1}(\{\mu\})$, and consequently $x(t;\overline{x})\rightarrow \Sigma^{D_1,\dots,D_p}_{d_1,\dots,d_p}\cap (I|_{\operatorname{Mrk}((D,I))})^{-1}(\{\mu\})$ for $t\rightarrow \infty$. Particularly, if $I|_{\operatorname{Mrk}((D,I))}$ is a proper map, then the set $(I|_{\operatorname{Mrk}((D,I))})^{-1}(\{\mu\})$ is compact in $\operatorname{Mrk}((D,I))$, dynamically invariant, and consequently for every $\overline{x}\in (I|_{\operatorname{Mrk}((D,I))})^{-1}(\{\mu\})$, we have that $x(t;\overline{x})\rightarrow \Sigma^{D_1,\dots,D_p}_{d_1,\dots,d_p}\cap (I|_{\operatorname{Mrk}((D,I))})^{-1}(\{\mu\})$ for $t\rightarrow \infty$.
\item [(f)] Suppose there exists $\mu$, a regular value of $I|_{\operatorname{Mrk}((D,I))}$, such that the intersection of some connected component of $(I|_{\operatorname{Mrk}((D,I))})^{-1}(\{\mu\})$ with the invariant set $\Sigma^{D_1,\dots, D_p}_{d_1,\dots, d_p}\cap\operatorname{Mrk}((D,I))$,  contains isolated points. Then each such a point is an asymptotically stable equilibrium point of $X_{\lambda}|_{(I|_{\operatorname{Mrk}((D,I))})^{-1}(\{\mu\})}$, for every $\lambda > 0$.
\end{itemize}
\end{theorem}
\begin{proof}
Let us define the smooth function $F:\operatorname{Mrk}((D,I))\rightarrow [0,\infty)$ given by
$$
F(x):=(D_1 (x)-d_1)^2 +\dots + (D_p (x)-d_p)^2, ~ (\forall) x\in\operatorname{Mrk}((D,I)).
$$
Let $\overline{x}\in \operatorname{Mrk}((D,I))$ be given, and let $t\in I_{\overline{x}}\subseteq\mathbb{R}\mapsto x(t;\overline{x})\in\operatorname{Mrk}((D,I))$ be the integral curve of the vector field $X_{\lambda}\in\mathfrak{X}(\operatorname{Mrk}((D,I)))$ such that $x(0;\overline{x})=\overline{x}$, where $I_{\overline{x}}\subseteq\mathbb{R}$ stands for the maximal domain of definition of the solution $x(\cdot;\overline{x})$.

Since $D=(D_1,\dots,D_p)$ and $I=(I_1,\dots,I_k)$ are first integrals of $X$, and the vector field $X_{0}^{\lambda}\in\mathfrak{X}(\operatorname{Mrk}((D,I)))$, $\lambda >0$ satisfies by construction the relations \eqref{SVFc}, then using the Theorem \eqref{thmdin}, it follows that the vector field $X_{\lambda}= X+ X_{0}^{\lambda}\in\mathfrak{X}(\operatorname{Mrk}((D,I)))$, $\lambda \geq 0$, satisfies the relations \eqref{SVFc} too. 

Hence, for each $\lambda \geq 0$, the vector field $X_{\lambda}\in\mathfrak{X}(\operatorname{Mrk}((D,I)))$ admits the vector type first integral $I|_{\operatorname{Mrk}((D,I))}$, and keeps dynamically invariant the set $\mathcal{S} \cap\operatorname{Mrk}((D,I)) = \Sigma^{D_1,\dots, D_p}_{d_1,\dots, d_p} \cap\operatorname{Mrk}((D,I)) =D^{-1}(\{(d_1,\dots,d_p)\})\cap\operatorname{Mrk}((D,I))$. Moreover, the following equalities hold true:
\begin{align}\label{lderc}
\begin{split}
\mathcal{L}_{X_{\lambda}}F &= \mathcal{L}_{X+ X_{0}^{\lambda}}F =\sum_{i=1}^{p}\mathcal{L}_{X+ X_{0}^{\lambda}}\left[(D_i - d_i)^2 \right] = \sum_{i=1}^{p}2(D_i - d_i)\mathcal{L}_{X+ X_{0}^{\lambda}}(D_i - d_i)\\
&= \sum_{i=1}^{p}2(D_i - d_i)\mathcal{L}_{X}(D_i - d_i)+\sum_{i=1}^{p}2(D_i - d_i)\mathcal{L}_{X_{0}^{\lambda}}(D_i - d_i)\\
&= \sum_{i=1}^{p}2(D_i - d_i)\mathcal{L}_{X}(D_i) + \sum_{i=1}^{p}2(D_i - d_i)\mathcal{L}_{X_{0}^{\lambda}}(D_i)\\
&= \sum_{i=1}^{p}2(D_i - d_i)\cdot 0+\sum_{i=1}^{p}2(D_i - d_i)\mathcal{L}_{X_{0}^{\lambda}}(D_i)\\
&=\sum_{i=1}^{p}2(D_i - d_i)(-\lambda) (D_i - d_i)= (-2\lambda)\sum_{i=1}^{p}(D_i - d_i)^2 \\
&= (-2\lambda) F.
\end{split}
\end{align}
Using the relation \eqref{lderc}, we obtain that
\begin{equation*}
\dfrac{\mathrm{d}}{\mathrm{d}t}F(x(t;\overline{x}))= (-2\lambda) F(x(t;\overline{x})), ~ (\forall)t\in I_{\overline{x}},
\end{equation*}
and hence 
\begin{equation}\label{grwc}
F(x(t;\overline{x}))= \exp(-2\lambda t) \cdot F(\overline{x}), ~ (\forall)t\in I_{\overline{x}}.
\end{equation}
Moreover, since the set of zeros of $F$ coincides with $\Sigma^{D_1,\dots,D_p}_{d_1,\dots,d_p}\cap \operatorname{Mrk}((D,I))$, the following sets equality holds true:
\begin{equation}\label{zerosetc}
\{x\in\operatorname{Mrk}((D,I)) : (\mathcal{L}_{X_{\lambda}}F)(x)=0\}=\Sigma^{D_1,\dots,D_p}_{d_1,\dots,d_p}\cap \operatorname{Mrk}((D,I)).
\end{equation}
\begin{itemize}
\item [(a)] Since by Theorem \eqref{mainTHMconserv} we have that $\mathcal{E}(X^{\lambda}_{0})=\Sigma^{D_1,\dots, D_p}_{d_1,\dots, d_p}\bigcap\operatorname{Mrk}((D,I))$, in order to complete the proof of the first statement, it is enough to show that $\mathcal{E}(X_{\lambda})=\mathcal{E}(X)\bigcap\mathcal{E}(X^{\lambda}_{0})$. We shall prove this equality by double inclusion. The inclusion $\mathcal{E}(X)\bigcap\mathcal{E}(X^{\lambda}_{0})\subseteq \mathcal{E}(X_{\lambda})$ is trivial since $X_{\lambda} =X+ X^{\lambda}_{0}$. In order to show the converse inclusion, $\mathcal{E}(X_{\lambda})\subseteq \mathcal{E}(X)\bigcap\mathcal{E}(X^{\lambda}_{0})$, let us pick some element $x_e \in \mathcal{E}(X_{\lambda})$. Then, since $X_{\lambda} =X+ X^{\lambda}_{0}$, $X_{\lambda}\in \mathfrak{X}(\operatorname{Mrk}((D,I)))$, it follows that $x_e \in \operatorname{Mrk}((D,I))$ and $X(x_e)+ X^{\lambda}_{0} (x_e)=0$. On the other hand, since $x_e \in \mathcal{E}(X_{\lambda})$, it follows that the integral curve of $X_{\lambda}$ starting from $x_e$, is constant, i.e., $x(t;x_e)=x_e$, for all $t\in\mathbb{R}$. Consequently, the relation \eqref{grwc} implies that $F(x_e)=\exp(-2\lambda t) \cdot F(x_e)$, for all $t\in\mathbb{R}$, and hence $F(x_e)=0$, which is in turn equivalent to $x_e \in \Sigma^{D_1,\dots, D_p}_{d_1,\dots, d_p}\bigcap\operatorname{Mrk}((D,I))$. As $\Sigma^{D_1,\dots, D_p}_{d_1,\dots, d_p}\bigcap\operatorname{Mrk}((D,I))=\mathcal{E}(X^{\lambda}_{0})$, it follows that $x_e \in \mathcal{E}(X^{\lambda}_{0})$, and consequently, since $X(x_e)+ X^{\lambda}_{0} (x_e)=0$ we get that $X(x_e) =0$, and hence $x_e \in \mathcal{E}(X)\bigcap\mathcal{E}(X^{\lambda}_{0})$. Since $x_e \in \mathcal{E}(X_{\lambda})$ was arbitrary chosen, it follows that $\mathcal{E}(X_{\lambda})\subseteq \mathcal{E}(X)\bigcap\mathcal{E}(X^{\lambda}_{0})$.
\item [(b)] In order to prove that $\Sigma^{D_1,\dots, D_p}_{d_1,\dots, d_p}\bigcap\operatorname{Mrk}((D,I))$ is a dynamically invariant set of each vector field $X_{\lambda}$, for every $\lambda \geq 0$, note that for $\lambda =0$ we obtain $X_{\lambda}=X$, and hence $\Sigma^{D_1,\dots, D_p}_{d_1,\dots, d_p}\bigcap\operatorname{Mrk}((D,I))$ is an invariant set of $X$, since $(D,I)$ is by definition a vector type first integral of $X$ and consequently both sets, $\Sigma^{D_1,\dots, D_p}_{d_1,\dots, d_p}$ and $\operatorname{Mrk}((D,I))$, are invariant. Hence, in order to complete the proof of this statement, it remains to show that $\Sigma^{D_1,\dots, D_p}_{d_1,\dots, d_p}\bigcap\operatorname{Mrk}((D,I))$ is a dynamically invariant set of each vector field $X_{\lambda}$, for every $\lambda > 0$. In order to do that, let us fix some $\lambda > 0$ and $\overline{x}\in \Sigma^{D_1,\dots, D_p}_{d_1,\dots, d_p}\bigcap\operatorname{Mrk}((D,I))$. We shall show that the integral curve of $X_{\lambda}$ starting from $\overline{x}$ at $t=0$, verifies that $x(t;\overline{x})\in \Sigma^{D_1,\dots, D_p}_{d_1,\dots, d_p}\bigcap\operatorname{Mrk}((D,I))$, for all $t\in I_{\overline{x}}$, where $I_{\overline{x}}\subseteq \mathbb{R}$ stands for the maximal domain of definition of the solution $x(\cdot;\overline{x})$. Using the relation \eqref{grwc}, we get that $F(x(t;\overline{x}))= \exp(-2\lambda t) \cdot F(\overline{x}), ~ (\forall)t\in I_{\overline{x}}$. Since $\overline{x}\in \Sigma^{D_1,\dots, D_p}_{d_1,\dots, d_p}\bigcap\operatorname{Mrk}((D,I))$ it follows that $F(\overline{x})=0$, and consequently we obtain that $F(x(t;\overline{x}))=0, ~ (\forall)t\in I_{\overline{x}}$, and so $x(t;\overline{x})\in \Sigma^{D_1,\dots, D_p}_{d_1,\dots, d_p}\bigcap\operatorname{Mrk}((D,I))$, for all $t\in I_{\overline{x}}$. Since $\overline{x}$ was arbitrary chosen in $\Sigma^{D_1,\dots, D_p}_{d_1,\dots, d_p}\bigcap\operatorname{Mrk}((D,I))$, it follows that $\Sigma^{D_1,\dots, D_p}_{d_1,\dots, d_p}\bigcap\operatorname{Mrk}((D,I))$ is a dynamically invariant set for $X_{\lambda}$. Moreover, since $I|_{\operatorname{Mrk}((D,I))}$ is a vector type first integral of $X_{\lambda}$ it follows that for each regular value $\mu$ of $I|_{\operatorname{Mrk}((D,I))}$,  
the set $\Sigma^{D_1,\dots,D_p}_{d_1,\dots,d_p}\cap (I|_{\operatorname{Mrk}((D,I))})^{-1}(\{\mu\})$ is a closed and dynamically invariant set of the vector field ${X_{\lambda}}$, since both sets, $\Sigma^{D_1,\dots,D_p}_{d_1,\dots,d_p}$ and $(I|_{\operatorname{Mrk}((D,I))})^{-1}(\{\mu\})$, are closed and invariant.
\item [(c)] In order to prove this item, pick an arbitrary element $\overline{x}\in (I|_{\operatorname{Mrk}((D,I))})^{-1}(\{\mu\})$, such that the set $\{x(t;\overline{x}): t\geq 0\}$ is bounded, where $t\mapsto x(t;\overline{x})$ stands for the integral curve of the vector field ${X_{\lambda}}|_{(I|_{\operatorname{Mrk}((D,I))})^{-1}(\{\mu\})}$, ($\lambda >0$), starting from $\overline{x}$ at $t=0$. We shall show now that the $\omega-$limit set $\omega(\overline{x})$ is a subset of $\Sigma^{D_1,\dots,D_p}_{d_1,\dots,d_p}\cap (I|_{\operatorname{Mrk}((D,I))})^{-1}(\{\mu\})$. Indeed, let $y\in\omega(\overline{x})$ be arbitrary chosen. Then, there exists a sequence $(t_{n})_{n\in\mathbb{N}}\subset [0,\infty)$, $\lim_{n\rightarrow \infty}t_n =\infty$, such that $\lim_{n\rightarrow \infty}x(t_n;\overline{x})=y$. Since the set $\{x(t;\overline{x}): t\geq 0\}$ is bounded, we get that $[0,\infty)\subset I_{\overline{x}}$, and hence the relation \eqref{grwc} implies that
\begin{equation}\label{frel}
F(x(t;\overline{x}))= \exp(-2\lambda t) \cdot F(\overline{x}), ~ (\forall)t\in [0,\infty).
\end{equation}
Consequently, for $t=t_n \geq 0$, $n\in\mathbb{N}$, the equality \eqref{frel} becomes
\begin{equation*}
F(x(t_n;\overline{x}))= \exp(-2\lambda t_n) \cdot F(\overline{x}), ~ (\forall)n\in \mathbb{N}.
\end{equation*}
Since $\lim_{n\rightarrow \infty} t_n =\infty $, $\lambda >0$, $\lim_{n\rightarrow \infty}x(t_n;\overline{x})=y$, and $F$ is continuous, we obtain that $F(y)=0$, and hence taking into account that the set of zeros of $F$ is $\Sigma^{D_1,\dots, D_p}_{d_1,\dots, d_p}\bigcap\operatorname{Mrk}((D,I))$, it follows that $y\in \Sigma^{D_1,\dots, D_p}_{d_1,\dots, d_p}\bigcap\operatorname{Mrk}((D,I))$. Since $\overline{x}\in(I|_{\operatorname{Mrk}((D,I))})^{-1}(\{\mu\})$, and $(I|_{\operatorname{Mrk}((D,I))})^{-1}(\{\mu\})\subset\operatorname{Mrk}((D,I))$ is closed and dynamically invariant, it follows that $y=\lim_{n\rightarrow \infty}x(t_n;\overline{x})\in(I|_{\operatorname{Mrk}((D,I))})^{-1}(\{\mu\})$, and consequently $y\in \Sigma^{D_1,\dots,D_p}_{d_1,\dots,d_p}\cap (I|_{\operatorname{Mrk}((D,I))})^{-1}(\{\mu\})$. As $y\in\omega(\overline{x})$ was arbitrary chosen, we obtain that $\omega(\overline{x})\subseteq \Sigma^{D_1,\dots,D_p}_{d_1,\dots,d_p}\cap (I|_{\operatorname{Mrk}((D,I))})^{-1}(\{\mu\})$. Since $x(t;\overline{x})\rightarrow \omega(\overline{x})\subseteq \Sigma^{D_1,\dots,D_p}_{d_1,\dots,d_p}\cap (I|_{\operatorname{Mrk}((D,I))})^{-1}(\{\mu\})$ for $t\rightarrow \infty$, it follows that $x(t;\overline{x})\rightarrow \Sigma^{D_1,\dots,D_p}_{d_1,\dots,d_p}\cap (I|_{\operatorname{Mrk}((D,I))})^{-1}(\{\mu\})$ for $t\rightarrow \infty$.
\item [(d)] Let $\gamma$ be an orbit of the vector field $X$ such that $\gamma \subset \Sigma^{D_1,\dots,D_p}_{d_1,\dots,d_p}\cap (I|_{\operatorname{Mrk}((D,I))})^{-1}(\{\mu\})$. Since $\Sigma^{D_1,\dots,D_p}_{d_1,\dots,d_p}\cap (I|_{\operatorname{Mrk}((D,I))})^{-1}(\{\mu\})\subseteq \Sigma^{D_1,\dots,D_p}_{d_1,\dots,d_p}\bigcap\operatorname{Mrk}((D,I))=\mathcal{E}(X^{\lambda}_{0})$ it follows that $X_{0}^{\lambda}(\gamma)= \{0\}$. As $\Sigma^{D_1,\dots,D_p}_{d_1,\dots,d_p}\cap (I|_{\operatorname{Mrk}((D,I))})^{-1}(\{\mu\})$ is a closed and dynamically invariant set of the vector field ${X_{\lambda}}=X+X_{0}^{\lambda}$, for every $\lambda \geq 0$, it follows that $\gamma$ is also an orbit of the same nature of the vector field ${X_{\lambda}}$, for every $\lambda \geq 0$. The rest of the proof is a direct consequence of $(c)$.
\item [(e)] The proof follows directly from $(c)$ since $K\subset (I|_{\operatorname{Mrk}((D,I))})^{-1}(\{\mu\})$ being a compact and positively invariant set, implies that for every $\overline{x}\in K$, the integral curve of ${X_{\lambda}}|_{(I|_{\operatorname{Mrk}((D,I))})^{-1}(\{\mu\})}$ starting from $\overline{x}$ at $t=0$, remains in $K$ for all $t\geq 0$, and hence the set $\{x(t;\overline{x}): t\geq 0\}$ is bounded.
\item [(f)] Let us denote by $C^{\mu}$ a connected component of $(I|_{\operatorname{Mrk}((D,I))})^{-1}(\{\mu\})$ whose intersection with the invariant set $\Sigma^{D_1,\dots, D_p}_{d_1,\dots, d_p}\cap\operatorname{Mrk}((D,I))$ contains isolated points. In order to complete the proof, it is enough to construct a strict Lyapunov function associated to each isolated equilibrium state $x_e \in\Sigma^{D_1,\dots,D_p}_{d_1,\dots,d_p}\cap C^{\mu}$. Note that the set $\Sigma^{D_1,\dots,D_p}_{d_1,\dots,d_p}\cap C^{\mu}$ is dynamically invariant, and since it contains isolated points, each isolated point must be an equilibrium point of the vector field ${X_{\lambda}}|_{(I|_{\operatorname{Mrk}((D,I))})^{-1}(\{\mu\})}$, for every $\lambda >0$. Let $x_e \in\Sigma^{D_1,\dots,D_p}_{d_1,\dots,d_p}\cap C^{\mu}$ be such an equilibrium point of ${X_{\lambda}}|_{(I|_{\operatorname{Mrk}((D,I))})^{-1}(\{\mu\})}$. Let us denote $C_{x_e}^{\mu}:=\Sigma^{D_1,\dots,D_p}_{d_1,\dots,d_p}\cap C^{\mu}$. Since $x_e \in C_{x_e}^{\mu}$ is an isolated point of $C_{x_e}^{\mu}$, there exists $U_{x_e} \subseteq \operatorname{Mrk}((D,I))$, an open neighborhood of $x_e$, such that $U_{x_e}\cap C_{x_e}^{\mu}=\{x_e \}$. By shrinking $U_{x_e}$ if necessary, one can suppose that    $U_{x_e}\cap\Sigma^{D_1,\dots,D_p}_{d_1,\dots,d_p}\cap(I|_{\operatorname{Mrk}((D,I))})^{-1}(\{\mu\})=\{x_e\}$.

Let us define now the smooth function $F: U_{x_e}\cap(I|_{\operatorname{Mrk}((D,I))})^{-1}(\{\mu\}) \rightarrow [0,\infty)$ given by
$$
F(x):=(D_1 (x)-d_1)^2 +\dots + (D_p (x)-d_p)^2, ~ (\forall) x\in U_{x_e}\cap(I|_{\operatorname{Mrk}((D,I))})^{-1}(\{\mu\}).
$$
Since by hypothesis we have that $$U_{x_e}\cap\Sigma^{D_1,\dots,D_p}_{d_1,\dots,d_p}\cap(I|_{\operatorname{Mrk}((D,I))})^{-1}(\{\mu\})=\{x_e\},$$ and the set of zeros of $F$ is given by $U_{x_e}\cap(I|_{\operatorname{Mrk}((D,I))})^{-1}(\{\mu\})\cap\Sigma^{D_1,\dots,D_p}_{d_1,\dots,d_p}$, it follows that $x_e$ is the unique solution of the equation $F(x)=0$ in $U_{x_e}\cap(I|_{\operatorname{Mrk}((D,I))})^{-1}(\{\mu\})$.
A similar relation to \eqref{lderc}, implies that
\begin{equation*}
(\mathcal{L}_{{X_{\lambda}}|_{(I|_{\operatorname{Mrk}((D,I))})^{-1}(\{\mu\})}}F) (x)= (-2\lambda)F(x), ~(\forall)x\in U_{x_e}\cap(I|_{\operatorname{Mrk}((D,I))})^{-1}(\{\mu\}).
\end{equation*}
Hence, we get that $F(x_e)=0$, $F(x)>0$, $(\mathcal{L}_{{X_{\lambda}}|_{(I|_{\operatorname{Mrk}((D,I))})^{-1}(\{\mu\})}}F) (x)<0$, for every $x\in \left( U_{x_e}\cap(I|_{\operatorname{Mrk}((D,I))})^{-1}(\{\mu\})\right)\setminus\{x_e\}$, and consequently $F$ is a strict Lyapunov function associated to the equilibrium point $x_e$.
\end{itemize}
\end{proof}

\section{Application to conservative dynamics with a prescribed foliated invariant set}

\medskip
The aim of this section is to present the correspondents of the main results of the previous section, in the case when the invariant set $\mathcal{S}\cap \operatorname{Mrk}((D,I))=\Sigma^{D_1,\dots, D_p}_{d_1,\dots, d_p}\cap \operatorname{Mrk}((D,I))$ is foliated by the level sets of regular values of the vector type first integral $D^{p^{\prime}\rightarrow }|_{\operatorname{Mrk}((D,I))}:=({D_{p^{\prime}+1}}|_{\operatorname{Mrk}((D,I))},\dots, {D_{p}}|_{\operatorname{Mrk}((D,I))})$, where $p^{\prime}$ is a natural number, such that $0<p^{\prime}<p$. More precisely, we consider a dynamical system $\mathfrak{S}$  generated by a smooth vector field $X\in\mathfrak{X}(U)$ defined on an open subset $U\subseteq\mathbb{R}^n$, which admits $k+p$ smooth first integrals ($1< k+p < n$, $k\geq 0$), $I_1,\dots, I_k, D_1,\dots, D_p \in\mathcal{C}^{\infty}(U,\mathbb{R})$. Let $\mathcal{S}$ be an invariant set of $\mathfrak{S}$, given by the level set of the vector type first integral $D$ corresponding to some regular (or singular) value $d:=(d_1,\dots,d_p)\in\operatorname{Im}(D)$, i.e., $\mathcal{S}= \Sigma^{D_1,\dots, D_p}_{d_1,\dots, d_p}:=D^{-1}(\{(d_1,\dots,d_p)\})$.

Let us fix some $p^{\prime}\in\mathbb{N}$ such that $0 < p^{\prime}\leq p$. Then using the Theorem \eqref{MTD}, a particular solution of the system of equations 
\begin{align}\label{SVFcp}
\begin{split}
&\mathcal{L}_{X}D_1 = (-\lambda)(D_1 - d_1), \dots, \mathcal{L}_{X}D_{p^{\prime}} = (-\lambda)(D_{p^{\prime}} - d_{p^{\prime}}),\\
&\mathcal{L}_{X}D_{p^{\prime}+1} =\dots =\mathcal{L}_{X}D_{p}=\mathcal{L}_{X}I_1 =\dots = \mathcal{L}_{X}I_k =0, 
\end{split}
\end{align}
is given by the vector field $X= X_{0}^{\lambda ;p^{\prime}} \in\mathfrak{X}(\operatorname{Mrk}((D,I)))$ defined by
\begin{equation}\label{xzerop}
X_{0}^{\lambda ;p^{\prime}}=\left\| \bigwedge_{i=1}^{p} \nabla D_i\wedge\bigwedge_{j=1}^{k} \nabla I_j \right\|_{k+p}^{-2}\cdot\sum_{i=1}^{p^{\prime}}(-1)^{n-i}(-\lambda) (D- d_i) \Theta_i, 
\end{equation}
where
$$
\Theta_i = \star\left[ \bigwedge_{j=1, j\neq i}^{p^{\prime}} \nabla D_j \wedge \bigwedge_{m={p^{\prime}+1}}^{p} \nabla D_{m}\wedge \bigwedge_{l=1}^{k} \nabla I_l  \wedge\star\left(\bigwedge_{j=1}^{p} \nabla D_j\wedge\bigwedge_{l=1}^{k} \nabla I_l \right)\right].
$$
Note that by construction, $X_{0}^{\lambda; p^{\prime}}\in\mathfrak{X}(\operatorname{Mrk}((D,I)))$ admits the vector type first integrals $I|_{\operatorname{Mrk}((D,I))}:=({I_1}|_{\operatorname{Mrk}((D,I))},\dots,{I_k}|_{\operatorname{Mrk}((D,I))})$ and $$D^{p^{\prime}\rightarrow }|_{\operatorname{Mrk}((D,I))}:=({D_{p^{\prime}+1}}|_{\operatorname{Mrk}((D,I))},\dots, {D_{p}}|_{\operatorname{Mrk}((D,I))}),$$ keeps $$(D_{1},\dots, D_{p^{\prime}})^{-1}(\{(d_1,\dots,d_{p^{\prime}})\})\cap \operatorname{Mrk}((D,I))$$ dynamically invariant, as well as the set $$(D_{1},\dots, D_{p^{\prime}})^{-1}(\{(d_1,\dots,d_{p^{\prime}})\})\cap (D^{p^{\prime}\rightarrow }|_{\operatorname{Mrk}((D,I))})^{-1}(\{\nu\})\cap (I|_{\operatorname{Mrk}((D,I))})^{-1}(\{\mu\}),$$ for each regular value $\mu$ of the vector type first integral $I|_{\operatorname{Mrk}((D,I))}$, and respectively for each regular value $\nu$ of the vector type first integral $D^{p^{\prime}\rightarrow }|_{\operatorname{Mrk}((D,I))}$. Let us state now a theorem which points out some of the main properties of the vector field $X_{0}^{\lambda; p^{\prime}}\in\mathfrak{X}(\operatorname{Mrk}((D,I)))$.

\begin{theorem}\label{mainTHMconservp}
Let $X_{0}^{\lambda; p^{\prime}}\in\mathfrak{X}(\operatorname{Mrk}((D,I)))$ be the vector field defined by the relation \eqref{xzerop}. Assuming that $(D_{1},\dots, D_{p^{\prime}})^{-1}(\{(d_1,\dots,d_{p^{\prime}})\})\cap \operatorname{Mrk}((D,I))\neq \emptyset$, the following statements hold true.
\begin{itemize}
\item [(a)] The set of the equilibrium states of the vector field $X_{0}^{\lambda; p^{\prime}}$ is given by $$\mathcal{E}(X_{0}^{\lambda; p^{\prime}})=(D_{1},\dots, D_{p^{\prime}})^{-1}(\{(d_1,\dots,d_{p^{\prime}})\})\cap \operatorname{Mrk}((D,I)).$$
\item [(b)] For each regular value $\mu$ of $I|_{\operatorname{Mrk}((D,I))}$ and respectively for each regular value $\nu$ of  $D^{p^{\prime}\rightarrow }|_{\operatorname{Mrk}((D,I))}$, the vector field ${X_{0}^{\lambda; p^{\prime}}}|_{(D^{p^{\prime}\rightarrow }|_{\operatorname{Mrk}((D,I))})^{-1}(\{\nu\})\cap (I|_{\operatorname{Mrk}((D,I))})^{-1}(\{\mu\})}$ admits the attracting set $$(D_{1},\dots, D_{p^{\prime}})^{-1}(\{(d_1,\dots,d_{p^{\prime}})\})\cap (D^{p^{\prime}\rightarrow }|_{\operatorname{Mrk}((D,I))})^{-1}(\{\nu\})\cap (I|_{\operatorname{Mrk}((D,I))})^{-1}(\{\mu\}).$$ More precisely, for every $\overline{x} \in (D^{p^{\prime}\rightarrow }|_{\operatorname{Mrk}((D,I))})^{-1}(\{\nu\})\cap (I|_{\operatorname{Mrk}((D,I))})^{-1}(\{\mu\})$, such that the set $\{x(t;\overline{x}): t\geq 0\}$ is bounded, $$x(t;\overline{x})\rightarrow (D_{1},\dots, D_{p^{\prime}})^{-1}(\{(d_1,\dots,d_{p^{\prime}})\})\cap (D^{p^{\prime}\rightarrow }|_{\operatorname{Mrk}((D,I))})^{-1}(\{\nu\})\cap (I|_{\operatorname{Mrk}((D,I))})^{-1}(\{\mu\}),$$ for $t\rightarrow \infty$.
\item [(b')] Suppose that $(d_{p^{\prime}+1},\dots, d_p)$ is a regular value of $(D_{p^{\prime}+1},\dots, D_p)$. Then for each regular value $\mu$ of $I|_{\operatorname{Mrk}((D,I))}$, the vector field $${X_{0}^{\lambda; p^{\prime}}}|_{(D^{p^{\prime}\rightarrow }|_{\operatorname{Mrk}((D,I))})^{-1}(\{(d_{p^{\prime}+1},\dots,d_p)\})\cap (I|_{\operatorname{Mrk}((D,I))})^{-1}(\{\mu\})}$$ admits the attracting set $\Sigma^{D_1,\dots,D_p}_{d_1,\dots,d_p}\cap (I|_{\operatorname{Mrk}((D,I))})^{-1}(\{\mu\})$. More precisely, for every $\overline{x} \in (D^{p^{\prime}\rightarrow }|_{\operatorname{Mrk}((D,I))})^{-1}(\{(d_{p^{\prime}+1},\dots,d_p)\})\cap (I|_{\operatorname{Mrk}((D,I))})^{-1}(\{\mu\})$, such that the set $\{x(t;\overline{x}): t\geq 0\}$ is bounded, $$x(t;\overline{x})\rightarrow \Sigma^{D_1,\dots,D_p}_{d_1,\dots,d_p}\cap (I|_{\operatorname{Mrk}((D,I))})^{-1}(\{\mu\}),$$ for $t\rightarrow \infty$.
\item [(c)] Let $\mu,\nu$ be arbitrary fixed regular values of $I|_{\operatorname{Mrk}((D,I))}$ and respectively $D^{p^{\prime}\rightarrow }|_{\operatorname{Mrk}((D,I))}$. Then for every positively invariant compact set $K\subset (D^{p^{\prime}\rightarrow }|_{\operatorname{Mrk}((D,I))})^{-1}(\{\nu\})\cap (I|_{\operatorname{Mrk}((D,I))})^{-1}(\{\mu\})$, and for every $\overline{x}\in K$ we have that $$\omega(\overline{x})\subseteq (D_{1},\dots, D_{p^{\prime}})^{-1}(\{(d_1,\dots,d_{p^{\prime}})\})\cap (D^{p^{\prime}\rightarrow }|_{\operatorname{Mrk}((D,I))})^{-1}(\{\nu\})\cap (I|_{\operatorname{Mrk}((D,I))})^{-1}(\{\mu\}),$$ and consequently $$x(t;\overline{x})\rightarrow (D_{1},\dots, D_{p^{\prime}})^{-1}(\{(d_1,\dots,d_{p^{\prime}})\})\cap (D^{p^{\prime}\rightarrow }|_{\operatorname{Mrk}((D,I))})^{-1}(\{\nu\})\cap (I|_{\operatorname{Mrk}((D,I))})^{-1}(\{\mu\})$$ for $t\rightarrow \infty$. Particularly, if $I|_{\operatorname{Mrk}((D,I))}$ or $D^{p^{\prime}\rightarrow }|_{\operatorname{Mrk}((D,I))}$ are proper maps, then the set $(D^{p^{\prime}\rightarrow }|_{\operatorname{Mrk}((D,I))})^{-1}(\{\nu\})\cap (I|_{\operatorname{Mrk}((D,I))})^{-1}(\{\mu\})$ is compact in $\operatorname{Mrk}((D,I))$, dynamically invariant, and consequently for every $\overline{x}\in (D^{p^{\prime}\rightarrow }|_{\operatorname{Mrk}((D,I))})^{-1}(\{\nu\})\cap (I|_{\operatorname{Mrk}((D,I))})^{-1}(\{\mu\})$, we have that $$x(t;\overline{x})\rightarrow (D_{1},\dots, D_{p^{\prime}})^{-1}(\{(d_1,\dots,d_{p^{\prime}})\})\cap (D^{p^{\prime}\rightarrow }|_{\operatorname{Mrk}((D,I))})^{-1}(\{\nu\})\cap (I|_{\operatorname{Mrk}((D,I))})^{-1}(\{\mu\})$$ for $t\rightarrow \infty$.
\item [(c')] Suppose that $(d_{p^{\prime}+1},\dots, d_p)$ is a regular value of $(D_{p^{\prime}+1},\dots, D_p)$. Let $\mu$ be an arbitrary fixed regular value of $I|_{\operatorname{Mrk}((D,I))}$. Then for every positively invariant compact set $$K\subset (D^{p^{\prime}\rightarrow }|_{\operatorname{Mrk}((D,I))})^{-1}(\{(d_{p^{\prime}+1},\dots,d_p)\})\cap (I|_{\operatorname{Mrk}((D,I))})^{-1}(\{\mu\}),$$ and for every $\overline{x}\in K$ we have that $\omega(\overline{x})\subseteq \Sigma^{D_1,\dots,D_p}_{d_1,\dots,d_p}\cap (I|_{\operatorname{Mrk}((D,I))})^{-1}(\{\mu\})$, and consequently $$x(t;\overline{x})\rightarrow \Sigma^{D_1,\dots,D_p}_{d_1,\dots,d_p}\cap (I|_{\operatorname{Mrk}((D,I))})^{-1}(\{\mu\})$$ for $t\rightarrow \infty$. Particularly, if $I|_{\operatorname{Mrk}((D,I))}$ or $D^{p^{\prime}\rightarrow }|_{\operatorname{Mrk}((D,I))}$ are proper maps, then the set $(D^{p^{\prime}\rightarrow }|_{\operatorname{Mrk}((D,I))})^{-1}(\{(d_{p^{\prime}+1},\dots,d_p)\})\cap (I|_{\operatorname{Mrk}((D,I))})^{-1}(\{\mu\})$ is compact in $\operatorname{Mrk}((D,I))$, dynamically invariant, and consequently for every $$\overline{x}\in (D^{p^{\prime}\rightarrow }|_{\operatorname{Mrk}((D,I))})^{-1}(\{(d_{p^{\prime}+1},\dots,d_p)\})\cap (I|_{\operatorname{Mrk}((D,I))})^{-1}(\{\mu\}),$$ we have that $$x(t;\overline{x})\rightarrow \Sigma^{D_1,\dots,D_p}_{d_1,\dots,d_p}\cap (I|_{\operatorname{Mrk}((D,I))})^{-1}(\{\mu\})$$ for $t\rightarrow \infty$.
\item [(d)] Suppose $\mu$ is a regular value of $I|_{\operatorname{Mrk}((D,I))}$ and $\nu$ is a regular value of $D^{p^{\prime}\rightarrow }|_{\operatorname{Mrk}((D,I))}$, such that  $$(D_{1},\dots, D_{p^{\prime}})^{-1}(\{(d_1,\dots,d_{p^{\prime}})\})\cap (D^{p^{\prime}\rightarrow }|_{\operatorname{Mrk}((D,I))})^{-1}(\{\nu\})\cap (I|_{\operatorname{Mrk}((D,I))})^{-1}(\{\mu\})$$ contains isolated points. Then each such a point $$\overline{x}\in(D_{1},\dots, D_{p^{\prime}})^{-1}(\{(d_1,\dots,d_{p^{\prime}})\})\cap (D^{p^{\prime}\rightarrow }|_{\operatorname{Mrk}((D,I))})^{-1}(\{\nu\})\cap (I|_{\operatorname{Mrk}((D,I))})^{-1}(\{\mu\})$$ is an asymptotically stable equilibrium point of the vector field $$X_{0}^{\lambda;p^{\prime}}|_{(D^{p^{\prime}\rightarrow }|_{\operatorname{Mrk}((D,I))})^{-1}(\{\nu\})\cap (I|_{\operatorname{Mrk}((D,I))})^{-1}(\{\mu\})}.$$
\item [(d')] Suppose that $(d_{p^{\prime}+1},\dots, d_p)$ is a regular value of $(D_{p^{\prime}+1},\dots, D_p)$. Suppose moreover that $\mu$ is a regular value of $I|_{\operatorname{Mrk}((D,I))}$, such that  $\Sigma^{D_1,\dots,D_p}_{d_1,\dots,d_p}\cap (I|_{\operatorname{Mrk}((D,I))})^{-1}(\{\mu\})$ contains isolated points. Then each such a point $\overline{x}\in\Sigma^{D_1,\dots,D_p}_{d_1,\dots,d_p}\cap (I|_{\operatorname{Mrk}((D,I))})^{-1}(\{\mu\})$ is an asymptotically stable equilibrium point of $$X_{0}^{\lambda;p^{\prime}}|_{(D^{p^{\prime}\rightarrow }|_{\operatorname{Mrk}((D,I))})^{-1}(\{(d_{p^{\prime}+1},\dots,d_p)\})\cap(I|_{\operatorname{Mrk}((D,I))})^{-1}(\{\mu\})}.$$
\end{itemize}
\end{theorem}
\begin{proof}
The proof follows mimetically those of Theorem \eqref{mainTHM} and Proposition \eqref{propok}, where the smooth function $F$ is replaced by $\widetilde{F}\in\mathcal{C}^{\infty}(\operatorname{Mrk}((D,I)),\mathbb{R})$
$$
\widetilde{F}(x):=(D_1 (x)-d_1)^2 +\dots + (D_{p^{\prime}} (x)-d_{p^{\prime}})^2, ~ (\forall) x\in \operatorname{Mrk}((D,I)),
$$
or, by its restriction to certain smooth manifolds of the type 
$$(D^{p^{\prime}\rightarrow }|_{\operatorname{Mrk}((D,I))})^{-1}(\{\nu\})\cap (I|_{\operatorname{Mrk}((D,I))})^{-1}(\{\mu\}).$$
\end{proof}

In the same way we did in the previous section, we construct a smooth family of dynamical systems $\{\mathfrak{S}_{\lambda}\}_{\lambda \geq 0}$ (defined on the open subset $\operatorname{Mrk}((D,I))\subseteq U$), such that $\mathfrak{S}_{0} =\mathfrak{S}|_{\operatorname{Mrk}((D,I))}$, and for all $\lambda >0$, the associated dynamical system, $\mathfrak{S}_{\lambda}$, admits the vector type first integrals $I|_{\operatorname{Mrk}((D,I))}$ and $D^{p^{\prime}\rightarrow }|_{\operatorname{Mrk}((D,I))}$, keeps dynamically invariant the set $(D_{1},\dots, D_{p^{\prime}})^{-1}(\{(d_1,\dots,d_{p^{\prime}})\})\cap\operatorname{Mrk}((D,I))$ and moreover, the invariant set $(D_{1},\dots, D_{p^{\prime}})^{-1}(\{(d_1,\dots,d_{p^{\prime}})\})\cap (D^{p^{\prime}\rightarrow }|_{\operatorname{Mrk}((D,I))})^{-1}(\{\nu\})\cap(I|_{\operatorname{Mrk}((D,I))})^{-1}(\{\mu\})$ (if not void) is an attracting set of $\mathfrak{S}_{\lambda}|_{(D^{p^{\prime}\rightarrow }|_{\operatorname{Mrk}((D,I))})^{-1}(\{\nu\})\cap(I|_{\operatorname{Mrk}((D,I))})^{-1}(\{\mu\})}$, for every regular value $\mu\in\operatorname{Im}(I|_{\operatorname{Mrk}((D,I))})$, and respectively for every regular value $\nu\in\operatorname{Im}(D^{p^{\prime}\rightarrow }|_{\operatorname{Mrk}((D,I))})$. Let us state now the correspondent of Theorem \eqref{THMcons1} in the above settings.

\begin{theorem}\label{THMcons2}
Let $X\in\mathfrak{X}(U)$ be a smooth vector field defined on an open subset $U\subseteq \mathbb{R}^n$, which admits $k+p$ smooth first integrals ($0< k+p < n$, $k\geq 0$), $I_1,\dots, I_k, D_1,\dots, D_p \in\mathcal{C}^{\infty}(U,\mathbb{R})$. Let $\mathcal{S}\subset U$ be a given closed and invariant set, defined as the level set of the vector type first integral $D:=(D_1,\dots, D_p)\in\mathcal{C}^{\infty}(U,\mathbb{R}^p)$ corresponding to some regular or singular value $d:=(d_1,\dots,d_p)\in\operatorname{Im}(D)$, i.e., $\mathcal{S}= \Sigma^{D_1,\dots, D_p}_{d_1,\dots, d_p}:=D^{-1}(\{(d_1,\dots,d_p)\})$. Let $p^{\prime}\in\mathbb{N}$ be a natural number such that $0 < p^{\prime}\leq p$.

Then, to each $\lambda \geq 0$, we associate a smooth vector field, $X_{\lambda;p^{\prime}}:=X+ X_{0}^{\lambda;p^{\prime}}$, defined on the open set $\operatorname{Mrk}((D,I))\subseteq U$, where the smooth vector field $X_{0}^{\lambda;p^{\prime}} \in \mathfrak{X}(\operatorname{Mrk}((D,I)))$ is given by
\begin{equation*}
X_{0}^{\lambda ;p^{\prime}}=\left\| \bigwedge_{i=1}^{p} \nabla D_i\wedge\bigwedge_{j=1}^{k} \nabla I_j \right\|_{k+p}^{-2}\cdot\sum_{i=1}^{p^{\prime}}(-1)^{n-i}(-\lambda) (D- d_i) \Theta_i, 
\end{equation*}
where
$$
\Theta_i = \star\left[ \bigwedge_{j=1, j\neq i}^{p^{\prime}} \nabla D_j \wedge \bigwedge_{m={p^{\prime}+1}}^{p} \nabla D_{m}\wedge \bigwedge_{l=1}^{k} \nabla I_l  \wedge\star\left(\bigwedge_{j=1}^{p} \nabla D_j\wedge\bigwedge_{l=1}^{k} \nabla I_l \right)\right].
$$
In the above settings, the following assertions hold true.
\begin{itemize}
\item[(a)] $\mathcal{E}(X_{\lambda;p^{\prime}})=\mathcal{E}(X)\cap\mathcal{E}(X^{\lambda;p^{\prime}}_{0})=\mathcal{E}(X)\cap(D_{1},\dots, D_{p^{\prime}})^{-1}(\{(d_1,\dots,d_{p^{\prime}})\})\cap \operatorname{Mrk}((D,I))$, where $\mathcal{E}(Z)$ stands for the set of equilibrium points of the vector field $Z$.
\item[(b)] The set $(D_{1},\dots, D_{p^{\prime}})^{-1}(\{(d_1,\dots,d_{p^{\prime}})\})\cap \operatorname{Mrk}((D,I))$ is a dynamically invariant set of the vector field $X_{\lambda;p^{\prime}}$, for every $\lambda \geq 0$. Moreover, for each regular value $\mu$ of $I|_{\operatorname{Mrk}((D,I))}$ and respectively for each regular value $\nu$ of $D^{p^{\prime}\rightarrow }|_{\operatorname{Mrk}((D,I))}$, the set $(D_{1},\dots, D_{p^{\prime}})^{-1}(\{(d_1,\dots,d_{p^{\prime}})\})\cap (D^{p^{\prime}\rightarrow }|_{\operatorname{Mrk}((D,I))})^{-1}(\{\nu\})\cap(I|_{\operatorname{Mrk}((D,I))})^{-1}(\{\mu\})$ is a closed and dynamically invariant set of the vector field ${X_{\lambda;p^{\prime}}}$, for every $\lambda \geq 0$.
\item [(c)] For each regular value $\mu$ of $I|_{\operatorname{Mrk}((D,I))}$ and respectively for each regular value $\nu$ of $D^{p^{\prime}\rightarrow }|_{\operatorname{Mrk}((D,I))}$, the vector field ${X_{\lambda;p^{\prime}}}|_{(D^{p^{\prime}\rightarrow }|_{\operatorname{Mrk}((D,I))})^{-1}(\{\nu\})\cap(I|_{\operatorname{Mrk}((D,I))})^{-1}(\{\mu\})}$, $\lambda >0$, admits the attracting set $$(D_{1},\dots, D_{p^{\prime}})^{-1}(\{(d_1,\dots,d_{p^{\prime}})\})\cap (D^{p^{\prime}\rightarrow }|_{\operatorname{Mrk}((D,I))})^{-1}(\{\nu\})\cap(I|_{\operatorname{Mrk}((D,I))})^{-1}(\{\mu\}).$$ More precisely, for every $\overline{x} \in (D^{p^{\prime}\rightarrow }|_{\operatorname{Mrk}((D,I))})^{-1}(\{\nu\})\cap(I|_{\operatorname{Mrk}((D,I))})^{-1}(\{\mu\})$, such that the set $\{x(t;\overline{x}): t\geq 0\}$ is bounded, $x(t;\overline{x})\rightarrow (D_{1},\dots, D_{p^{\prime}})^{-1}(\{(d_1,\dots,d_{p^{\prime}})\})\cap (D^{p^{\prime}\rightarrow }|_{\operatorname{Mrk}((D,I))})^{-1}(\{\nu\})\cap(I|_{\operatorname{Mrk}((D,I))})^{-1}(\{\mu\})$ for $t\rightarrow \infty$. 
\item [(d)] Suppose there exists $\mu$, a regular value of $I|_{\operatorname{Mrk}((D,I))}$, and respectively $\nu$, a regular value of $D^{p^{\prime}\rightarrow }|_{\operatorname{Mrk}((D,I))}$, such that the intersection of some connected component of $$(D^{p^{\prime}\rightarrow }|_{\operatorname{Mrk}((D,I))})^{-1}(\{\nu\})\cap(I|_{\operatorname{Mrk}((D,I))})^{-1}(\{\mu\})$$ with the invariant set $(D_{1},\dots, D_{p^{\prime}})^{-1}(\{(d_1,\dots,d_{p^{\prime}})\})\cap\operatorname{Mrk}((D,I))$, contains a unique orbit $\gamma$ of $X$. Then $\gamma$ preserves its nature as an orbit of $${X_{\lambda;p^{\prime}}}|_{(D^{p^{\prime}\rightarrow }|_{\operatorname{Mrk}((D,I))})^{-1}(\{\nu\})\cap(I|_{\operatorname{Mrk}((D,I))})^{-1}(\{\mu\})}$$ (for every $\lambda >0$), and moreover, attracts every bounded positive orbit of $${X_{\lambda;p^{\prime}}}|_{(D^{p^{\prime}\rightarrow }|_{\operatorname{Mrk}((D,I))})^{-1}(\{\nu\})\cap(I|_{\operatorname{Mrk}((D,I))})^{-1}(\{\mu\})},$$ sharing the same connected component.
\item [(e)] Let $\mu$ be an arbitrary fixed regular value of the vector type first integral $I|_{\operatorname{Mrk}((D,I))}$, and let $\nu$ be an arbitrary fixed regular value of $D^{p^{\prime}\rightarrow }|_{\operatorname{Mrk}((D,I))}$. Then for every positively invariant compact set $K\subset (D^{p^{\prime}\rightarrow }|_{\operatorname{Mrk}((D,I))})^{-1}(\{\nu\})\cap(I|_{\operatorname{Mrk}((D,I))})^{-1}(\{\mu\})$, and for every $\overline{x}\in K$ we have that $\omega(\overline{x})\subseteq (D_{1},\dots, D_{p^{\prime}})^{-1}(\{(d_1,\dots,d_{p^{\prime}})\})\cap (D^{p^{\prime}\rightarrow }|_{\operatorname{Mrk}((D,I))})^{-1}(\{\nu\})\cap(I|_{\operatorname{Mrk}((D,I))})^{-1}(\{\mu\})$, and consequently $x(t;\overline{x})$ approaches $(D_{1},\dots, D_{p^{\prime}})^{-1}(\{(d_1,\dots,d_{p^{\prime}})\})\cap (D^{p^{\prime}\rightarrow }|_{\operatorname{Mrk}((D,I))})^{-1}(\{\nu\})\cap(I|_{\operatorname{Mrk}((D,I))})^{-1}(\{\mu\})$ for $t\rightarrow \infty$. Particularly, if $I|_{\operatorname{Mrk}((D,I))}$ or $D^{p^{\prime}\rightarrow }|_{\operatorname{Mrk}((D,I))}$ are proper maps, then the set $(D^{p^{\prime}\rightarrow }|_{\operatorname{Mrk}((D,I))})^{-1}(\{\nu\})\cap(I|_{\operatorname{Mrk}((D,I))})^{-1}(\{\mu\})$ is compact in $\operatorname{Mrk}((D,I))$, dynamically invariant, and consequently for every $\overline{x}\in (D^{p^{\prime}\rightarrow }|_{\operatorname{Mrk}((D,I))})^{-1}(\{\nu\})\cap(I|_{\operatorname{Mrk}((D,I))})^{-1}(\{\mu\})$, we have that $x(t;\overline{x})$ approaches $(D_{1},\dots, D_{p^{\prime}})^{-1}(\{(d_1,\dots,d_{p^{\prime}})\})\cap (D^{p^{\prime}\rightarrow }|_{\operatorname{Mrk}((D,I))})^{-1}(\{\nu\})\cap(I|_{\operatorname{Mrk}((D,I))})^{-1}(\{\mu\})$ for $t\rightarrow \infty$.
\item [(f)] Suppose there exists $\mu$ and $\nu$, two regular values of $I|_{\operatorname{Mrk}((D,I))}$ and respectively $D^{p^{\prime}\rightarrow }|_{\operatorname{Mrk}((D,I))}$, such that the intersection of some connected component of $$(D^{p^{\prime}\rightarrow }|_{\operatorname{Mrk}((D,I))})^{-1}(\{\nu\})\cap(I|_{\operatorname{Mrk}((D,I))})^{-1}(\{\mu\})$$ with the invariant set $(D_{1},\dots, D_{p^{\prime}})^{-1}(\{(d_1,\dots,d_{p^{\prime}})\})\cap\operatorname{Mrk}((D,I))$, contains isolated points. Then each such a point is an asymptotically stable equilibrium point of the vector field $X_{\lambda;p^{\prime}}|_{(D^{p^{\prime}\rightarrow }|_{\operatorname{Mrk}((D,I))})^{-1}(\{\nu\})\cap(I|_{\operatorname{Mrk}((D,I))})^{-1}(\{\mu\})}$, for every $\lambda > 0$.
\end{itemize}
\end{theorem}
\begin{proof}
The proof follows mimetically that of Theorem \eqref{THMcons1}.
\end{proof}

\begin{remark}
In the hypothesis of Theorem \eqref{THMcons2}, if $\nu=(d_{p^{\prime}+1},\dots,d_p)$ is a regular value of $(D_{p^{\prime}+1},\dots, D_p)$, then $$(D_{1},\dots, D_{p^{\prime}})^{-1}(\{(d_1,\dots,d_{p^{\prime}})\})\cap (D^{p^{\prime}\rightarrow }|_{\operatorname{Mrk}((D,I))})^{-1}(\{\nu\})=\Sigma^{D_1,\dots,D_p}_{d_1,\dots,d_p}\cap\operatorname{Mrk}((D,I)),$$ and moreover, the items $(b),(c),(d),(e)$ and $(f)$ of Theorem \eqref{THMcons2}, become as follows:
\begin{itemize}
\item[(b')] The set $(D_{1},\dots, D_{p^{\prime}})^{-1}(\{(d_1,\dots,d_{p^{\prime}})\})\cap \operatorname{Mrk}((D,I))$ is a dynamically invariant set of the vector field $X_{\lambda;p^{\prime}}$, for every $\lambda \geq 0$. Moreover, for each regular value $\mu$ of $I|_{\operatorname{Mrk}((D,I))}$, the set $\Sigma^{D_1,\dots,D_p}_{d_1,\dots,d_p}\cap(I|_{\operatorname{Mrk}((D,I))})^{-1}(\{\mu\})$ is a closed and dynamically invariant set of the vector field ${X_{\lambda;p^{\prime}}}$, for every $\lambda \geq 0$.
\item [(c')] For each regular value $\mu$ of $I|_{\operatorname{Mrk}((D,I))}$, the vector field $${X_{\lambda;p^{\prime}}}|_{(D^{p^{\prime}\rightarrow }|_{\operatorname{Mrk}((D,I))})^{-1}(\{(d_{p^{\prime}+1},\dots,d_p)\})\cap(I|_{\operatorname{Mrk}((D,I))})^{-1}(\{\mu\})}, ~\lambda >0,$$ admits the attracting set $\Sigma^{D_1,\dots,D_p}_{d_1,\dots,d_p}\cap(I|_{\operatorname{Mrk}((D,I))})^{-1}(\{\mu\}).$ More precisely, for every $\overline{x} \in (D^{p^{\prime}\rightarrow }|_{\operatorname{Mrk}((D,I))})^{-1}(\{(d_{p^{\prime}+1},\dots,d_p)\})\cap(I|_{\operatorname{Mrk}((D,I))})^{-1}(\{\mu\})$, such that the set $\{x(t;\overline{x}): t\geq 0\}$ is bounded, $x(t;\overline{x})\rightarrow \Sigma^{D_1,\dots,D_p}_{d_1,\dots,d_p}\cap(I|_{\operatorname{Mrk}((D,I))})^{-1}(\{\mu\})$ for $t\rightarrow \infty$.
\item [(d')] Suppose there exists $\mu$, a regular value of $I|_{\operatorname{Mrk}((D,I))}$, such that the intersection of some connected component of $$(D^{p^{\prime}\rightarrow }|_{\operatorname{Mrk}((D,I))})^{-1}(\{(d_{p^{\prime}+1},\dots,d_p)\})\cap(I|_{\operatorname{Mrk}((D,I))})^{-1}(\{\mu\})$$ with the invariant set $(D_{1},\dots, D_{p^{\prime}})^{-1}(\{(d_1,\dots,d_{p^{\prime}})\})\cap\operatorname{Mrk}((D,I))$, contains a unique orbit $\gamma$ of $X$. Then $\gamma$ preserves its nature as an orbit of $${X_{\lambda;p^{\prime}}}|_{(D^{p^{\prime}\rightarrow }|_{\operatorname{Mrk}((D,I))})^{-1}(\{(d_{p^{\prime}+1},\dots,d_p)\})\cap(I|_{\operatorname{Mrk}((D,I))})^{-1}(\{\mu\})}$$ (for every $\lambda >0$), and moreover, attracts every bounded positive orbit of $${X_{\lambda;p^{\prime}}}|_{(D^{p^{\prime}\rightarrow }|_{\operatorname{Mrk}((D,I))})^{-1}(\{(d_{p^{\prime}+1},\dots,d_p)\})\cap(I|_{\operatorname{Mrk}((D,I))})^{-1}(\{\mu\})},$$ sharing the same connected component.
\item [(e')] Let $\mu$ be an arbitrary fixed regular value of the vector type first integral $I|_{\operatorname{Mrk}((D,I))}$. Then for every positively invariant compact set $$K\subset (D^{p^{\prime}\rightarrow }|_{\operatorname{Mrk}((D,I))})^{-1}(\{(d_{p^{\prime}+1},\dots,d_p)\})\cap(I|_{\operatorname{Mrk}((D,I))})^{-1}(\{\mu\}),$$ and for every $\overline{x}\in K$ we have that $\omega(\overline{x})\subseteq \Sigma^{D_1,\dots,D_p}_{d_1,\dots,d_p}\cap(I|_{\operatorname{Mrk}((D,I))})^{-1}(\{\mu\})$, and consequently $x(t;\overline{x})$ approaches $\Sigma^{D_1,\dots,D_p}_{d_1,\dots,d_p}\cap(I|_{\operatorname{Mrk}((D,I))})^{-1}(\{\mu\})$ for $t\rightarrow \infty$. Particularly, if $I|_{\operatorname{Mrk}((D,I))}$ or $D^{p^{\prime}\rightarrow }|_{\operatorname{Mrk}((D,I))}$ are proper maps, then the set $$(D^{p^{\prime}\rightarrow }|_{\operatorname{Mrk}((D,I))})^{-1}(\{(d_{p^{\prime}+1},\dots,d_p)\})\cap(I|_{\operatorname{Mrk}((D,I))})^{-1}(\{\mu\})$$ is compact in $\operatorname{Mrk}((D,I))$, dynamically invariant, and consequently for every $\overline{x}\in (D^{p^{\prime}\rightarrow }|_{\operatorname{Mrk}((D,I))})^{-1}(\{(d_{p^{\prime}+1},\dots,d_p)\})\cap(I|_{\operatorname{Mrk}((D,I))})^{-1}(\{\mu\})$, we have that $x(t;\overline{x})$ approaches $\Sigma^{D_1,\dots,D_p}_{d_1,\dots,d_p}\cap(I|_{\operatorname{Mrk}((D,I))})^{-1}(\{\mu\})$ for $t\rightarrow \infty$.
\item [(f')] Suppose there exists $\mu$ a regular value of $I|_{\operatorname{Mrk}((D,I))}$ such that the intersection of some connected component of $$(D^{p^{\prime}\rightarrow }|_{\operatorname{Mrk}((D,I))})^{-1}(\{(d_{p^{\prime}+1},\dots,d_p)\})\cap(I|_{\operatorname{Mrk}((D,I))})^{-1}(\{\mu\})$$ with the invariant set $(D_{1},\dots, D_{p^{\prime}})^{-1}(\{(d_1,\dots,d_{p^{\prime}})\})\cap\operatorname{Mrk}((D,I))$, contains isolated points. Then each such a point is an asymptotically stable equilibrium point of the vector field $X_{\lambda;p^{\prime}}|_{(D^{p^{\prime}\rightarrow }|_{\operatorname{Mrk}((D,I))})^{-1}(\{(d_{p^{\prime}+1},\dots,d_p)\})\cap(I|_{\operatorname{Mrk}((D,I))})^{-1}(\{\mu\})}$, for every $\lambda > 0$.
\end{itemize}
\end{remark}


\bigskip
\bigskip

\noindent {\sc R.M. Tudoran}\\
West University of Timi\c soara\\
Faculty of Mathematics and Computer Science\\
Department of Mathematics\\
Blvd. Vasile P\^arvan, No. 4\\
300223 - Timi\c soara, Rom\^ania.\\
E-mail: {\sf tudoran@math.uvt.ro}\\
\medskip

\end{document}